\documentclass[a4paper,12pt]{article}

\usepackage[bbgreekl]{mathbbol}
\usepackage{mathrsfs}
\usepackage{graphicx}
\usepackage{amsmath}
\usepackage{amsfonts}
\usepackage{amssymb}
\usepackage{amsthm}
\usepackage{color}
\usepackage{cancel}
\usepackage{subfigure}
\usepackage{appendix}
\usepackage{bm}
\usepackage{comment}
\usepackage{epstopdf}
\usepackage{multirow}
\usepackage{colortbl}

\usepackage{url}

\usepackage[colorlinks,linkcolor=blue,anchorcolor=green,citecolor=blue]{hyperref}

\newtheorem{theorem}{Theorem}[section]
\newtheorem{lemma}{Lemma}[section]
\newtheorem{remark}{Remark}[section]

\renewcommand{\theequation}{\arabic{section}.\arabic{equation}}

\newcommand{\bvec}[1]{\mathbf{#1}}

\newcommand{\vh}{\bvec{h}}
\newcommand{\vi}{\bvec{i}}

\newcommand{\vn}{\bvec{n}}

\newcommand{\vp}{\bvec{p}}

\newcommand{\vr}{\bvec{r}}
\newcommand{\vs}{\bvec{s}}

\newcommand{\vR}{\bvec{R}}

\newcommand{\R}{\mathbb{R}}

\newcommand{\node}{\mathrm{\Xi}}

\newcommand{\pro}{\pi_{p,\node}}
\newcommand{\Node}{\bm{\node}}
\newcommand{\Spline}{S_{\vp}(\Node)}
\newcommand{\Splinei}{S_{\vp^{(i)}}(\Node^{(i)})}
\newcommand{\Pro}{\pi_{\vp,\Node}}

\newcommand{\Ne}{N_{\rm e}}
\newcommand{\dd}{{\rm d}}
\newcommand{\dg}{{\rm DG}}
\newcommand{\facet}{\mathcal{F}_{\mathcal{I}}}

\graphicspath{{./figs/}}

\title{Numerical Analysis of Multi-patch Discontinuous Galerkin Isogeometric Method for Full-potential Electronic Structure Calculations
}
\author{Xiaoxu Li\thanks{{\it xiaoxuli@bnu.edu.cn}.
Faculty of Arts and Sciences, Beijing Normal University, No.18 Jinfeng Road, Tangjiawan, Zhuhai 519087, Guangdong, China.} 
~and Xucheng Meng \thanks{{\it Corresponding. xcmeng@bnu.edu.cn}.
{Research Center for Mathematics, Beijing Normal University, Zhuhai 519087, Guangdong, China \& Guangdong Provincial Key Laboratory of Interdisciplinary Research and Application for Data Science, 
BNU-HKBU United International College, Zhuhai 519087, Guangdong, China}}
}
\date{}

\begin{document}
\maketitle

\begin{abstract}
In this paper, we study the multi-patch discontinuous Galerkin isogeometric (DG-IGA) approximations for full-potential electronic structure calculations.
We decompose the physical domain into several subdomains, represent each part of the wavefunction separately using B-spline basis functions, possibly with different degrees, on varying mesh sizes, and then combine them by DG methods.
We also provide a rigorous {\em a priori} error analysis of the DG-IGA approximations for linear eigenvalue problems.
Furthermore, this work offers a unified analysis framework for the DG-IGA method applied to a class of elliptic eigenvalue problems.
Finally, we present several numerical experiments to verify our theoretical results.

\end{abstract}


\section{Introduction}\label{sec-intro} 
\setcounter{equation}{0}

Electronic structure calculations describe the energies and distributions of electrons, which are fundamental to many fields, including materials science, biochemistry, solid-state physics, and surface physics. 
Among the various electronic structure models, the Kohn-Sham density functional theory (DFT) \cite{martin05} offers the best compromise between precision and computational cost.
For an $\Ne$-electron system with $M$ nuclei, each with charge $Z_k$ located at ${\bf R}_k\in\mathbb{R}^3$~$(k=1,\cdots,M)$,
Kohn-Sham DFT leads to the following nonlinear eigenvalue problems
\begin{eqnarray}\label{eq:eigen}
H u_i=\lambda_i u_i,\quad\lambda_1\leq\lambda_2\leq\cdots\leq\lambda_{\Ne},
\end{eqnarray}
where the Kohn-Sham Hamiltonian is given by
\begin{eqnarray*}
H=-\frac{1}{2}\Delta+V_{\rm ext}+V_{\rm H}[\rho]+V_{\rm xc}[\rho]. 
\end{eqnarray*}
Note that the Hamiltonian $H$ depends on the first $\Ne$ eigenfunctions.
Here, $\displaystyle V_{\rm ext}({\vr})=-\sum_{k=1}^M\frac{Z_k}{|{\vr}-{\vR}_k|}$  is the external potential generated by nuclear attraction,
$\displaystyle V_{\rm H} [\rho]=\int_{\mathbb{R}^3}\frac{\rho({\vr'})}{|\cdot-{\vr'}|} \dd {\vr'}$ and
$ V_{\rm xc}[\rho]$ are the so-called Hartree potential and the exchange-correlation potential, respectively,
with the electron density $\displaystyle \rho(\vr)=\sum_{i=1}^{\Ne}|u_i(\vr)|^2$.
A self-consistent field (SCF) iteration algorithm \cite{cances00,lebris03} is commonly used to solve these nonlinear problems. 
In each iteration, a Hamiltonian $H$ is constructed from trial electronic states $\{\tilde{u}_1,\cdots,\tilde{u}_{\Ne}\}$, and a linear eigenvalue problem is solved to obtain the low-lying eigenvalues and corresponding eigenfunctions. The process continues until the electronic states reach self-consistency.
The efficiency of this algorithm is primarily influenced by the discretization of the Hamiltonian, the self-consistent iteration, and the linear eigensolver. In this paper, we will focus on the discretization method.

In practical simulations, the eigenvalue problem \eqref{eq:eigen} is typically computed in a bounded domain with homogeneous Dirichlet boundary condition, as the eigenfunctions of \eqref{eq:eigen} decay exponentially as $|\vr|\rightarrow \infty$ (see, e.g., \cite{almbladh1985exact}).
For eigenvalue problems with singular potentials in full-potential calculations, methods based on discretization using linear combinations of atomic orbitals, such as Gaussians and Slater-type orbitals \cite{herring40, martin05}, have been widely used in quantum chemistry (see \cite{bachmayr14, chen15b} for their numerical analysis).
The main advantage of atomic orbitals is their efficiency, as the number of basis functions required is much fewer compared to other methods. 
However, this efficiency comes at the cost of a lack of systematic convergence \cite{batcho2000computational}.
There has also been a number of studies on finite element methods (FEM) in electronic structure calculations \cite{beck2000real,fang2012kohn, pask1999real,pask2001finite, tsuchida1995electronic}.
Although FEM requires more degrees of freedom than methods using atomic orbitals, it results in a sparse algebraic eigenvalue problem and can be implemented with adaptive refinement techniques \cite{BAO20124967,bao2013numerical,chen2014adaptive,dai2008convergence}.

Isogeometric analysis (IGA), introduced by Hughes, Cottrell and Bazilevs in \cite{hughes2005isogeometric}, is a collection of methods (commonly referred to as {\it isogeometric methods}) that use B-splines, or some of their extensions such as NURBS (non-uniform rational B-splines), as basis functions to build approximation spaces for solving partial differential equations numerically (see, e.g., \cite{da2014mathematical} for a mathematical review).
Compared to FEM, IGA methods offer a more flexible and accurate representation of complex geometries and, more importantly, provide smoother basis functions \cite{bazilevs2006isogeometric,da2014mathematical}.
However, due to their global construction, IGA methods are not suitable for local refinement. 
In recent years, there has been increasing interest in applying IGA to electronic structure calculations \cite{cimrman2018convergence,cimrman2018isogeometric,temizer2020nurbs}.

For the full-potential calculation with singular potentials, it has been shown that the eigenfunction exhibits cusps at the nuclei positions \cite{fournais02,fournais04, fournais07,hoffmann01}.
In this paper, we decompose the physical domain into several subdomains (referred to as {\em patches}), represent each part of the wavefunction separately using B-spline basis functions with (possibly) different degrees on varying mesh sizes, and then patch them together using DG methods.


The DG framework has been widely used to obtain numerical solutions of partial differential equations, and has been extensively studied theoretically in numerous works (see, e.g., \cite{arnold00,babuska73,buffa08,harriman03,suli00} and references cited therein).
Multi-patch IGA combined with DG methods has also been explored for various problems \cite{CHAN201822,duvigneau2018isogeometric,langer2015analysis,MICHOSKI2016658,moore2019space}.
For electronic structure calculations, Lin et al. \cite{lin12, zhang17} construct basis functions adaptively from the local environment and patch them together in the global domain using DG methods. In \cite{li2019discontinuous}, one of our authors employs different basis functions in different regions, which are then connected using the DG method.

In this paper, we present an {\em a priori} error analysis of our DG-IGA approximations for linear eigenvalue problems. Our analysis is closely related to the techniques used in \cite{antonietti06, harriman03, osborn75, suli00}.
The contributions of this work are twofold. First, this is the first work to utilize the discontinuous Galerkin isogeometric method for the electronic structure calculations, along with an {\em a priori} error estimate (see Theorem \ref{them:convergence-rate}). Second, this work provides a unified analysis framework for the discontinuous Galerkin isogeometric method applied to a class of elliptic eigenvalue problems, which, to our knowledge, is novel.

\vskip 0.2cm
{\bf Outline.} 
The remainder of this paper is organized as follows. 
In Section \ref{sec:bspline}, we briefly introduce the B-spline basis functions and the projection error on the approximation space. 
In Section \ref{sec:DG-approximation}, we present a DG-IGA discretization scheme, along with a numerical analysis of convergence and an {\it a priori} error estimate for the DG-IGA approximation.
We display several numerical examples to verify our theoretical results in Section \ref{sec-numerical}.
Finally, we offer concluding remarks in Section \ref{sec-conclusions}, and collect the detailed proofs of our results in the appendices.

\vskip 0.2cm

{\bf Notations.} 
Let the computational domain be $\Omega=[-D/2,D/2]^d~(D>0)$ with the physical dimension $d=1,2,3$ and
$L^2(\Omega)$ be the space of square integrable functions with the inner product denoted by $(\cdot,\cdot)$.
We introduce the Sobolev space 
$H^s(\Omega) =\{v \in L^2(\Omega) ~:~ D^{\bm{\alpha}} v \in L^2(\Omega), ~ \text{for} ~ 0 \leq |\bm{\alpha}| \leq s  \}$ with the equipped norm 
$\|v\|_{H^s(\Omega)} = \left( \sum_{0 \leq |\bm{\alpha}| \leq s} \|D^{\bm{\alpha}} v\|^2_{L^2(\Omega)} \right)^{1/2}$, where $\bm{\alpha} = (\alpha_1,\cdots,\alpha_d)$ is a multi-index with non-negative integers $\alpha_1,\cdots,\alpha_d$, and the partial derivative $D^{\bm{\alpha}} := \partial^{|\bm{\alpha}|}/\partial \vr^{\bm{\alpha}}$ with $|\bm{\alpha}| = \alpha_1 +\cdots + \alpha_d$.
We also denote $H^s_0(\Omega)= \{v \in H^s(\Omega) ~:~ v=0 ~ \text{on} ~ \partial\Omega \}$ and set $H^0(\Omega)=L^2(\Omega)$.
The notation $\mathscr{L}(X)$ is used to represent the set of bounded operators from the Banach space $X$ to itself with the norm 
$\|T\|_{\mathscr{L}(X)} = \sup_{f\in X} \|Tf\|_{X}/\|f\|_{X}$.

Throughout this paper, the symbol $C$ will denote a generic positive constant that may change from one line to the next, which will always be independent of the approximation parameters under consideration.
The dependence of $C$ will normally be clear from the context or stated explicitly.

\section{B-spline approximations}
\label{sec:bspline}
\setcounter{equation}{0}

In this section, we present a brief overview of the B-spline approximations, which provides some vital information for the formulation and discussion of multi-patch isogeometric analysis.
We first introduce the univariate B-spline basis and the projection onto the corresponding function space. These definitions are then extended to the multi-dimensional case using tensor products, and we conclude with an error estimate for multivariate B-spline approximations.

For positive integers $p$ and $n$, let us define a vector $\mathrm{\Xi} = \left \{0=\xi_1,\ldots,\xi_{n+p+1}=1\right\}$ with a non-decreasing sequence of real numbers in the parametric domain $[0,1]$, which is called a {\em knot} vector on the unit interval.
In this work, we will only work with {\it open} knot vectors, which means that the first $p+1$ knots in $\node$ are equal to $0$, and the last $p+1$ knots are equal to $1$.
Given $\mathrm{\Xi}$ with  degree $p \geq 1$ and the number of basis functions $n$, the univariate B-spline basis functions are defined by the Cox-de Boor recursion formula \cite{schumaker2007spline}
\begin{align}
B_{i,0}(\xi)  & = \left\{
\begin{array}{ll}\nonumber
1 \qquad \qquad &  \text{if} \quad \xi_{i} \leq \xi < \xi_{i+1}, \\[1ex]
0 & \text{otherwise}, \\
\end{array}\right. \\[1ex]
B_{i,p}(\xi)  & = \frac{\xi-\xi_{i}}{\xi_{i+p}
- \xi_{i}}B_{i,p-1}(\xi)
+\frac{\xi_{i+p+1}-\xi}{\xi_{i+p+1}-\xi_{i+1}}B_{i+1,p-1}(\xi) \qquad p\geq 1,
\end{align}
where any division by zero is defined to be zero. 
In general, a basis function of the degree $p$ is $p-m$ times continuously differentiable across a knot with the multiplicity $m$. For simplicity, we assume that all the internal knots have the multiplicity $m = 1$, then B-splines of degree $p$ are globally $(p-1)-$continuously differentiable. 
In addition, each $B_{i,p}$ is non-negative and supported in $(\xi_{i}, \xi_{i+p+1})$.

We then define the projection $\displaystyle \pro: L^2([0,1]) \rightarrow \text{span}\{B_{i,p},~i=1,\cdots,n\}$ by
\begin{eqnarray}
\pro f := \sum_{i=1}^n \mathcal{D}_{i,p}(f) B_{i,p}  \qquad \forall~f\in L^2([0,1]),
\end{eqnarray}
where the dual basis functions  $\mathcal{D}_{i,p}$ \cite {da2014mathematical,schumaker2007spline} satisfy 
\begin{align}\nonumber
&\mathcal{D}_{i,p}(B_{j,p}) = \delta_{ij} \qquad i,j=1,2,\cdots,n \qquad \text{with} \\[1ex]
&\mathcal{D}_{1,p}(f) = f(0), \quad \mathcal{D}_{n,p}(f) = f(1).
\end{align}

In general higher-dimensional cases, the B-spline basis functions are tensor product of the univariate B-spline basis functions. 
Let $ \mathrm{\Xi}_{\alpha} = \left \{0=\xi_{1,\alpha},\ldots,\xi_{n_{\alpha}+p_{\alpha}+1,\alpha}=1\right\}$ be the knot vectors for every direction $\alpha = 1,\ldots,d$, and $\Node=\node_1\times\cdots\times\node_d$.
Let $\mathbf{i} = (i_1,\ldots,i_{d})$, $\mathbf{p} = (p_1,\ldots,p_{d})$ and the set 
$\mathcal{I}=\{\mathbf{i} = (i_1,\ldots,i_{d}) : 1\leq i_\alpha \leq n_\alpha\} $ be multi-indices.  
Then the tensor product B-spline basis functions are defined by
\begin{equation}
B_{\vi,\vp}(\bm{\xi}) := \prod\limits_{\alpha=1}^{d}B_{i_\alpha,p_\alpha}(\xi_{\alpha}),
\end{equation}
where $\bm{\xi} = (\xi_1,\ldots,\xi_{d}) \in \widehat{\Omega} := [0,1]^d$. 
Associated with the knot vectors $\Node$, the domain $[0,1]^d$ is partitioned into rectangles or rectangular prisms, forming a mesh $\mathcal{Q}_h$. For any $Q \in \mathcal{Q}_h$, we define $h_Q = \text{diam}(Q)$, and the global mesh size is given by $h = \max\{ h_Q, ~Q \in \mathcal{Q}_h \}$.

The multivariate spline space in the parametric domain $\widehat{\Omega}$ is given by 
\begin{equation}
\Spline = \text{span}\{B_{\mathbf{i},\mathbf{p}}, ~\mathbf{i} \in  \mathcal{I} \}.
\end{equation}
It is natural to extend the projection to multi-dimensional case by
\begin{eqnarray}
\Pro := \pi_{\vp_1,\Node_1} \otimes \cdots \otimes \pi_{\vp_d,\Node_d}: L^2([0,1]^d) \rightarrow \Spline.
\end{eqnarray}
From \cite{bazilevs2006isogeometric, da2014mathematical}, and \cite[Chapter 12]{schumaker2007spline}, we have the following approximation error estimate.

\begin{lemma}[{\bf projection error}]
\label{lemma:bspline-err}
Let $0\leq k\leq s$ and $\displaystyle p=\min \{p_i\}$.
If $f\in H^{s}([0,1]^d)$, then there exists a positive constant $C$ such that
\begin{eqnarray}
	\|f- \Pro f\|_{H^k([0,1]^d)}  \leq C
	h^{t-k} \|f\|_{H^t([0,1]^d)}
\end{eqnarray}
with $t=\min\{p+1,s\}$.
\end{lemma}

\section{DG multi-patch isogeometric method}
\label{sec:DG-approximation} 
\setcounter{equation}{0}
In this section, we construct a DG isogeometric discretization scheme using distinct sizes of mesh in different subdomains.
We provide an {\it a priori} error analysis of the numerical approximations.
Our analysis is composed of three steps: 
First, we estimate the best approximation errors in each subdomain separately;
We then give an error estimate for the DG approximation of the corresponding source problem;
Finally, we derive an error estimate for the eigenvalue problem.
Note that the errors generated by numerical quadratures and linear algebraic
solvers are not considered in this paper, which deserve separate investigations.

As a model problem, we  consider the following Schr\"{o}dinger-type linear eigenvalue problem, which can be viewed as a
linearization of \eqref{eq:eigen}: Find $\lambda\in\mathbb{R}$ and $0\neq u\in H_{0}^1(\Omega)$ such that $\|u\|_{L^2(\Omega)}=1$ and
\begin{eqnarray}
	\label{model-eq}
	a(u,v)=\lambda(u,v) \qquad\forall~v\in H_{0}^1(\Omega),
\end{eqnarray}
where the bilinear form $a(\cdot,\cdot):H_{0}^1(\Omega)\times H_{0}^1(\Omega)\rightarrow\mathbb{C}$ is given by
\begin{eqnarray}
	\label{bilinear-a}
	a(u,v)=\frac{1}{2}\int_{\Omega}\nabla u\cdot\nabla v\, {\rm d}{\vr} +\int_{\Omega} Vuv\,{\rm d
    }{\vr}
\end{eqnarray}
with the potential $V\in L^2(\Omega)$.
The existence and uniqueness of the variational problem \eqref{model-eq} 
follows the well-known Lax-Milgram Theorem. 

It was shown in \cite{fournais02,fournais04,fournais07} that the exact electron densities are analytic away from the nuclei and satisfy certain cusp conditions at the nuclei positions (the regularity is not better than $H^{5/2}(\Omega)$ around the singularities for $d=3$). 
It can be seen from Lemma \ref{lemma:bspline-err} that the low regularity of wavefunction limits the convergence rate of high-order B-spline approximations.
However, the wavefunction is smooth except near the nuclei positions. 
This motivates us to decompose the computational domain into different subdomains and represent the wavefunctions separately in each region.

\subsection{Multi-patch discretization}
\label{sec-discretization}

To achieve the distinct representations of wavefunctions in different regions, the physical domain $\Omega$ is divided into $N$ non-overlapping subdomains $\{ \Omega_i \}_{i=1}^N$, referred to as {\em patches} such that $\overline{\Omega} = \bigcup_{i=1}^N \overline{\Omega}_i$ and $ \Omega_i \cap \Omega_j = \emptyset$ for $i \neq j$.
We require that the singularities (i.e., the nuclei positions) are contained within the interior of some patches.
For simplicity, we will focus our discussion on a single atom located at the origin, though the algorithms and analysis can be easily generalized to multi-atom systems.
Throughout this paper, we denote the interior faces shared by two patches as $F_{ij} = \partial \Omega_i \cap \partial \Omega_j ~(i \neq j)$. The collection of all such interior and external faces is denoted by $\facet$ and $\mathcal{F}_{\mathcal{E}}$, respectively.

Each patch is the image of an associated geometrical mapping $G_i$, such that $\Omega_i = G_i(\widehat{\Omega}) $ for $i = 1, \ldots, N$. 
In general, some regularity condition is assumed for the geometrical mapping $G_i$, such as the bi-Lipschitz homeomorphism (see, e.g., \cite{da2014mathematical}). 
However, in our model problem, the computational domain can be simply divided into rectangular subdomains, allowing the geometrical maps to be simple re-scaling functions.
For $1\leq i\leq N$, the map is given by
\begin{eqnarray}
\label{rescale_map}
G_i(\bm{\widehat{\xi}}) = A_i \bm{\widehat{\xi}} \qquad \forall~ \bm{\hat{\xi}}\in\widehat{\Omega}\equiv[0,1]^d,
\end{eqnarray}
where $A_i\in\R^{d\times d}$ is a diagonal matrix.

Let $\mathcal{Q}_{h_i,i}$ be the mesh for the patch $\Omega_i$ with $1\leq i\leq N$.
For each patch, we assume that the underlying mesh $\mathcal{Q}_{h_i,i}$ is quasi-uniform, that is, there exists a constant $C_{\rm m}\geq 1$ such that
\begin{eqnarray}\nonumber
\label{quasi-uniform-mesh}
h_Q  \leq h_i \leq C_{\rm m} h_Q \qquad\forall~Q \in \mathcal{Q}_{h_i, i},
\end{eqnarray}
where $h_i = \max\{h_Q, ~Q\in\mathcal{Q}_{h_i,i}\}$ is the mesh size of patch $\Omega_i$.
In the following, we will use $\vh=(h_1,  \cdots, h_N)$, $h_{\rm min} = \min(h_1, \cdots, h_N)$ and $h_{\rm max} = \max(h_1, \cdots, h_N) $ to represent mesh sizes of patches in the computational domain $\Omega$, the minimum and maximum mesh sizes, respectively.
Let $\vp^{(i)}$ and $\Node^{(i)}$ be the degree of polynomials and the knot vectors of the patch $\Omega_i$ for $i=1,\cdots,N$.
We can define the spline space in the physical subdomain $\Omega_i$ by 
\begin{equation}
V^{(i)}_{h_i} = \{f \circ G^{-1}_i : f\in\Splinei \},
\end{equation}
and the corresponding projection $\Pi_i: L^2(\Omega_i)\rightarrow V^{(i)}_{h_i}$ by
\begin{eqnarray}
\Pi_{i} f := \pi_{\vp^{(i)},\Node^{(i)}}(f\circ G_i) \circ G^{-1}_i \qquad \forall~ f\in L^2(\Omega_i).
\end{eqnarray}

With a slight abuse of the notation, we denote $\displaystyle p_i=\min \{p^{(i)}_1,\cdots,p^{(i)}_d\}$ hereafter, unless otherwise stated.
Lemma \ref{lemma:bspline-err} can be directly generalized to the following estimate on the physical subdomains.  The generalization for our model problem with the re-scaling maps \eqref{rescale_map} is straightforward, while the generalization for patches decomposition with more complex geometrical maps is provided in \cite{bazilevs2006isogeometric}.

\begin{lemma}[{\bf projection error on physical subdomain}]
\label{lemma:multi-patch-projection-err}
Let $0\leq k\leq s_i$ and $1\leq i\leq N$.
If $f\in H^{s_i}(\Omega_i)$, then there exists a positive constant $C$ depending on $G_i$ such that
\begin{eqnarray}	
\label{multi-patch-projection-err}
\|f- \Pi_i f\|_{H^k(\Omega_i)}  \leq C
h_i^{t_i-k} \|f\|_{H^{t_i}(\Omega_i)}
\end{eqnarray}
with $t_i=\min\{p_i+1,s_i\}$.
\end{lemma}

Before ending this subsection, we present the following trace inequality in the spline space, which plays an important role for error analysis in the next subsection.
The proof of this inequality can be found in \cite{langer2015analysis}. 

\begin{lemma}[{\bf trace inequality}]
\label{lemma:trace-inequality}
Let $q\in\R^+$ and $1\leq i\leq N$. 
For all $v\in V_{h_i}^{(i)}$ and $F_{ij}\subset \partial \Omega_i~(i\neq j)$, there exists a positive constant $C$, which depends on $G_i$ and the quasi-uniform constant $C_{\rm m}$ but not on $h_i$, such that
\begin{eqnarray}
\label{trace-inequality-eq}
\|v\|^q_{L^q(F_{ij})} \leq \frac{C}{h_i} \|v\|^q_{L^q(\Omega_i)}.
\end{eqnarray}
\end{lemma}

\subsection{DG approximations of source problem}
\label{subsec-bvP}

In this subsection, we will discuss the DG discretization for the source problem and our analysis is related to the framework in \cite{antonietti06, li2019discontinuous}.

Based on the decomposition of the physical domain, we first introduce the following Sobolev space with $\vs=(s_1, \cdots, s_N)$ and $s_i \geq 1$ for $1\leq i\leq N$
\begin{eqnarray*}
\widetilde{H}^{\vs}(\Omega) = 
\left\{ v\in H^1_0(\Omega) : v|_{\Omega_i}\in H^{s_i}(\Omega_i)~ \text{for}~ 1\leq i\leq N\right\}
\end{eqnarray*}
with the equipped norm 
\begin{eqnarray*}
\|v\|_{\widetilde{H}^{\vs}(\Omega)} := \sum_{i=1}^N \|v\|_{H^{s_i}(\Omega_i)}.
\end{eqnarray*}

%
For vector-valued function ${\bf w}$ and scalar-valued function $u$ which are not continuous on the internal surface $F_{ij}$ for some $i\ne j$, we define the jumps by
\begin{eqnarray*}
[{\bf w}]={\bf w}^+ \cdot {\bf n}^+ +{\bf w}^- \cdot {\bf n}^-,\quad\quad [u]=u^+{\bf n}^+
+u^-{\bf n}^-,
\end{eqnarray*}
and the averages by
\begin{eqnarray*}
\{{\bf w}\}=\frac{1}{2}({\bf w}^+ +{\bf w}^-),\quad\quad \{u\}=\frac{1}{2}(u^++u^-),
\end{eqnarray*}
where ${\bf w}^{\pm}$ and $u^{\pm}$ are traces of ${\bf w}$ and $u$ on $F_{ij}$ taken from both sides, ${\bf n}^{\pm}$ are the normal unit vectors.

We represent the discrete approximation space in the whole domain $\Omega$ by
\begin{equation}
V_{\vh} = \{f\in L^2(\Omega): f\big|_{\Omega_i} \in V^{(i)}_{h_i} ~\text{and}~f|_{\mathcal{F}_{\mathcal{E}}}=0\}.
\end{equation}
Then the bilinear form $a^{\rm DG}(\cdot,\cdot): (V_{\vh} \cup H^1_0(\Omega)) \times (V_{\vh} \cup H^1_0(\Omega)) \rightarrow\mathbb{C}$ is given by
\begin{eqnarray}\label{eq-bilinear-DG}
\nonumber
a^{\rm DG}(u,v) &=& \sum_{i=1}^N \int_{\Omega_i}\left(\frac{1}{2}\nabla u\cdot\nabla v+Vuv\right){\rm d}{\vr}-\sum_{F_{ij}\in \facet} \frac{1}{2}\int_{F_{ij}}\{\nabla u\}\cdot[v] \dd s \\[1ex]
&& 
-\sum_{F_{ij}\in \facet} \frac{1}{2}\int_{F_{ij}}\{\nabla v\}\cdot[u] \dd s
+ C_\sigma \sum_{F_{ij}\in \facet} \left(\frac{1}{h_i}+\frac{1}{h_j}\right) \int_{F_{ij}} [u]\cdot[v] \dd s,
\end{eqnarray}
where $C_{\sigma}$ is a sufficiently large discontinuity-penalization parameter  independent of the discretization.
Note that there are many other types of DG formulations (see, e.g., \cite{antonietti06,arnold00}), 
and \eqref{eq-bilinear-DG} is the classical symmetric interior penalty (SIP) method \cite{antonietti06,arnold82,lin12}.

We further define the broken Sobolev space
\begin{align*}
H_{\delta}(\Omega) = \left\{v\in L^2(\Omega):
~v|_{\Omega_i}\in H^1(\Omega_i) ~\text{for}~ 1\leq i\leq N\right\} 
\end{align*}
equipped with the following DG-norm
\begin{eqnarray}\label{eq-norm-DG}
\|u\|^2_{\rm DG} = \sum_{i=1}^N \|u\|^2_{H^1(\Omega_i)} + C_\sigma \sum_{F_{ij}\in\facet} \left(\frac{1}{h_i}+\frac{1}{h_j}\right) \left\|[u] \right\|^2_{L^2(F_{ij})}.
\end{eqnarray}
We then provide the following two lemmas to demonstrate the boundedness and coerciveness of the bilinear form with respect to the DG-norm, whose proofs are given in \ref{append-boundness} and \ref{append-coercive}, respectively.

\begin{lemma}[{\bf boundness}]
\label{lem:boundness}
There exists a positive constant $C$ such that
\begin{eqnarray}\label{boundness}
	a^{\rm DG}(u, v) \leq C  \|u\|_{\rm DG}  \|v\|_{\rm DG}
	\qquad \forall~u, v\in V_{\vh}.
\end{eqnarray}
\end{lemma}

\begin{lemma}[{\bf coerciveness}]
\label{lem:coercive}
If $C_{\sigma}$ is sufficiently large, then there exist positive constants $\alpha$ and $\beta$ such that
\begin{eqnarray}\label{coercive}
a^{\rm DG}(u,u)\geq \alpha\|u\|^2_{\rm DG}-\beta\|u\|^2_{L^2(\Omega)}
\qquad \forall~u\in V_{\vh}.
\end{eqnarray}
\end{lemma}

For simplicity, we can take $\beta=0$. Note that $a^{\rm DG}_{\beta}(u,v) =a^{\rm DG}(u,v)+\beta(u,v)$ makes this true for $\beta>0$.

Define the solution operator
\begin{eqnarray*}
T:L^2(\Omega)\rightarrow H_0^1(\Omega) \qquad \text{such that} \quad a(Tf,v)=(f,v)\quad\forall~v\in H_0^1(\Omega),
\end{eqnarray*}
and the approximation operator
\begin{eqnarray*}
T^{\rm DG}:L^2(\Omega)\rightarrow V_\vh
\qquad \text{such that} \quad a^{\rm DG}(T^{\rm DG}f,v)=(f,v)\quad\forall~v\in V_\vh.
\end{eqnarray*}
The following theorem provides the error estimate in the DG-norm for the approximation of the solution operator, whose proof is provided in \ref{append-T-approximation}.

\begin{theorem}[{\bf error estimate in DG-norm}]
\label{thm:T-approximate}
Let $C_{\sigma}$ be sufficiently large and $\vs=(s_1,\cdots,s_N)$. 
If $Tf\in \widetilde{H}^{\vs}(\Omega)$
for $f\in L^2(\Omega)$,
then there exists a positive constant $C$ such that
\begin{eqnarray}
\label{rate-T-tildeT}
\|(T-T^{\rm DG})f\|_{\rm DG} \leq C \sum_{i=1}^N h_i^{t_i-1} \|Tf\|_{H^{t_i}(\Omega_i)}
\end{eqnarray}
with $t_i=\min\{p_i+1,s_i\}$.
\end{theorem}

Note that the regularity estimate 
\begin{eqnarray}\label{regularity-estimate-eq}
	\|Tf\|_{H^2(\Omega)} \leq C \|f\|_{L^2(\Omega)}
\end{eqnarray}
holds for all $f\in L^2(\Omega)$ \cite{evans2022partial}. 
We can apply the standard duality techniques to bound the error in $L^2$-norm, which is stated in the following theorem.
We point out that our DG approximation can be viewed as a special variant of non-conforming scheme.
The main idea behind the $L^2$-norm estimate follows similar techniques used in non-conforming finite elements (see, e.g., \cite{brenner08, sun2016finite}), and we provide the proof in \ref{append-error-L2} for the sake of completeness.

\begin{theorem}[{\bf error estimate in $L^2$-norm}]
\label{thm:error_L2}
Under the conditions in Theorem \ref{thm:T-approximate}, there exists a positive constant $C$ such that
\begin{eqnarray}
	\|(T-T^{\rm DG})f\|_{L^2(\Omega)} \leq C h_{\rm max} \|(T-T^{\rm DG})f\|_{\rm DG}.
\end{eqnarray}
\end{theorem}

\subsection{DG approximations of eigenvalue problem}
\label{subsec-eigen}

We construct DG methods for the linear eigenvalue problem \eqref{model-eq}: Find $\lambda^{\rm DG}
\in\mathbb{R}$ and $u^{\rm DG}\in V_{\vh}$, such that 
$\|u^{\rm DG}\|_{L^2(\Omega)} =1$ and
\begin{eqnarray}\label{eq-eigen-DG}
a^{\rm DG}(u^{\rm DG},v)=\lambda^{\rm DG}(u^{\rm DG},v)
\qquad\forall~v\in V_{\vh}.
\end{eqnarray}
Note that \eqref{model-eq} and \eqref{eq-eigen-DG} are equivalent to $\lambda Tu=u$
and $\lambda^{\rm DG}T^{\rm DG}u^{\rm DG}=u^{\rm DG}$, respectively.
For a nonzero eigenvalue $\lambda$ of \eqref{model-eq}, we denote by $M(\lambda)$ the eigenspace spanned by the corresponding eigenfunctions.

According to Theorem \ref{thm:T-approximate} and regularity estimate \eqref{regularity-estimate-eq}, we can obtain the convergence of $\{T^{\dg}\}: L^2(\Omega) \rightarrow L^2(\Omega)$, i.e. $\|T-T^{\dg}\|_{\mathscr{L}(L^2(\Omega))} \rightarrow 0$ as $h_{\rm max}\rightarrow 0$,
which can lead to the convergence of eigenpair \cite{osborn75, yang2008order,yang2010eigenvalue}.
The following theorem relates the error of the eigenvalue problem to that of the corresponding source problem, whose proof is given in \ref{append-eigen-source-bound}.

\begin{theorem}[{\bf convergence of eigenvalue problem}]
\label{thm:eigen-source-bound}
Let nonzero $\lambda_k$ be the $k$-th eigenvalue of \eqref{model-eq} and $(\lambda_k^{\rm DG}, u_k^{\rm DG})$ be the $k$-th eigenpair of \eqref{eq-eigen-DG}. 
Then $\lambda_k\rightarrow \lambda^{\dg}_k$ as $h_{\rm max}\rightarrow 0$ and there exists $u_k\in M(\lambda_k)$ with $\|u_k\|_{L^2(\Omega)}=1$ such that 
\begin{align}
\label{source-eigenvalue-a}
\| u_k^{\dg}-u_k \|_{L^2(\Omega)} &\leq C \| (T-T^{\dg})u_k \|_{L^2(\Omega)},  \\[1ex]\label{source-eigenvalue-b}
|\lambda^{\dg}_k - \lambda_k| & \leq C \|(T-T^{\dg})u_k\|_{\dg} \inf_{v\in V_{\vh}} \|u_k-v\|_{\dg} + R_1, \\[1ex]\label{source-eigenvalue-c}
\| u_k^{\dg}-u_k \|_{\dg} &= |\lambda_k| \cdot \| (T-T^{\dg})u_k \|_{\dg} + R_2,
\end{align}
where $\displaystyle |R_1| \leq C \| (T-T^{\dg})u_k \|^2_{L^2(\Omega)}$ and  $\displaystyle |R_2| \leq C \| (T-T^{\dg})u_k \|_{L^2(\Omega)}$.
\end{theorem}

\begin{remark}[{\bf repeated eigenvalues}]
We emphasize that the convergence result works not only for the case of single eigenvalue,	but also for general cases of multiple eigenvalues with $\lambda_k=\lambda_{k+1}=\cdots=\lambda_{k+q-1}$.
\end{remark}

The following theorem is the main result of this paper, which gives {\it a priori} error estimate of the eigenvalue problem for DG approximations.

\begin{theorem}[{\bf error estimate of eigenvalue problem}]
\label{them:convergence-rate}
Let nonzero $\lambda_k$ be the $k$-th eigenvalue of \eqref{model-eq} and $(\lambda_k^{\rm DG}, u_k^{\rm DG})$ be the $k$-th eigenpair of \eqref{eq-eigen-DG}. 
Suppose $M(\lambda_k) \subset \widetilde{H}^{\vs}(\Omega)$ with $\vs=(s_1,\cdots,s_N)$. 
If $C_{\sigma}$ is sufficiently large, then there exists $u_k\in M(\lambda_k)$ with $\|u_k\|_{L^2(\Omega)}=1$ such that 
\begin{align}\label{convergence-rate-eq}
\|u_k^{\rm DG}-u_k\|_{\rm DG}  &\leq C \sum_{i=1}^N  h_i^{t_i-1} \|u_k\|_{H^{t_i}(\Omega_i)}, \\[1ex]
|\lambda_k^{\rm DG}-\lambda_k| &\leq C \|u_k^{\rm DG}-u_k\|^2_{\rm DG}, \\[1ex]
\|u_k^{\rm DG}-u_k\|_{L^2(\Omega)} &\leq C h_{\rm max} \|u_k^{\rm DG}-u_k\|_{\rm DG},
\end{align}
where $t_i=\min\{p_i+1,s_i\}$ and 
the positive constant $C$ depends on the eigenvalue $\lambda_k$. 
\end{theorem}

\begin{proof}
The error estimate is a direct corollary of Theorem \ref{thm:T-approximate},  \ref{thm:error_L2} and  \ref{thm:eigen-source-bound}.
\end{proof}

\begin{remark}[{\bf nonlinear eigenvalue problem}]
Within the framework of Kohn-Sham density functional theory, one needs to solve the nonlinear eigenvalue problem \eqref{eq:eigen}  with the SCF iteration. 
Using our DG discretizations, the linear eigenvalue problem \eqref{eq-eigen-DG} is solved at each iteration step and complex mixing schemes such as Roothaan, level-shifting, and DIIS algorithms (see, e.g., \cite{cances00,lebris03}) are used to achieve convergence.
	
If the exchange-correlation potential $V_{\rm xc}$ is sufficiently smooth and the trial states (from previous DG approximations) $\{\tilde{u}_1, \cdots, \tilde{u}_{\Ne}\}\in (V_{\vh})^{\Ne}$, then we have from arguments similar to those in \cite{flad08} that the eigenfunctions $\{u_i\}_{i=1,\cdots,\Ne}$ of the Kohn-Sham Hamiltonian $H$ in \eqref{eq:eigen} are smooth except at the positions of the nuclei.
This regularity together with the analysis in Theorem \ref{them:convergence-rate} gives the convergence rates for the DG approximations of the (linear) eigenvalue problem in each SCF iteration step.
	
Note that we have not provided an {\it a priori} error estimate for approximations of nonlinear eigenvalue problems, but only for linearized equations in SCF iterations.
We refer to \cite{cances12,chen13} for numerical analysis of nonlinear eigenvalue problems.
\end{remark}

\begin{remark}[{\bf multiscale refinement}]
\label{remark:multiscale-refinement}
We point out that if the mesh sizes of different subdomains are refined at the same rate, where there exists a positive constant $C$ not depending on the mesh size $\vh$ such that $h_{\rm max} \le C h_{\rm min}$, the convergence rate in \eqref{convergence-rate-eq} would be restricted by the lowest regularity in subdomains.

From a multiscale discretization perspective, the error bound in Theorem \ref{them:convergence-rate} suggests a better refinement strategy that corresponds to the local regularity of the eigenfunctions. 
The lower the local regularity in a subdomain, the faster the mesh refinement.
The ratio of mesh refinement in different regions is determined by the local regularity.

Let us take a simple case with two subdomains as an illustrative example.
We assume the eigenfunction $u\in \widetilde{H}^{(s_1, s_2)}(\Omega)$ with $1\leq s_1<s_2$, i.e., $u\in H^{s_1}(\Omega_1)$ and $u\in H^{s_2}(\Omega_2)$,
and the polynomial degrees $p_1$ and $p_2$ are both larger than $s_2-1$.
If the uniform mesh size $h_1=h_2=h_{\rm max}$ is used in the whole computational domain, the error in \eqref{convergence-rate-eq} would be $\mathcal{O}(h_{\rm max}^{2(s_1-1)})$.
On the other hand, we can use a coarse mesh in $\Omega_2$ while applying a fine mesh in $\Omega_1$ with $h_1=\mathcal{O}(h_2^{\frac{s_2-1}{s_1-1}})$, which gives rise to a higher convergence rate $\mathcal{O}(h_{\rm max}^{2(s_2-1)})$.

In theory, the convergence rate can be improved from the order determined by the lowest local regularity to that determined by the highest local regularity among subdomains. In practice, however, if there is a significant disparity in local regularity across different subdomains, refining the mesh to a sufficiently small scale becomes challenging, hindering the achievement of the optimal convergence order. Nevertheless, this multiscale mesh refinement approach can significantly accelerate convergence, especially when handling substantial differences in local regularity across regions.
\end{remark}

\section{Numerical results}
\label{sec-numerical} 
\setcounter{equation}{0}
\setcounter{figure}{0}

In this section, we will present several numerical experiments in electronic structure calculations. The purpose of this section is not only to validate our theoretical results but also to demonstrate the effectiveness of the DG-IGA method. 
In the following, the eigensolver EIGIFP  \cite{money2005algorithm} is adopted to solve the generalized eigenvalue problem, and the degrees of B-splines used to test the convergence rate of the eigenvalues and the corresponding eigenfunctions are $p=1$ and $p=2$, and the reference solution is obtained by the DG-IGA method on a sufficiently dense mesh with cubic B-spline basis functions, unless otherwise explicitly stated. All the simulations are performed on a PC with  Intel Core i9-12900K processor and 128 GB RAM, and the source codes can be found in \cite{DG_IGA_Eigenvalue_Codes}.

\vskip 0.5cm

\noindent{\bf Example 1 (2D linear problem for a single-atom system).}
Consider the two-dimensional linear eigenvalue problem: Find $\lambda\in\mathbb{R}$ and $u\in H^1_0(\Omega)$ with $\|u\|_{L^2(\Omega)}=1$ such that
\begin{eqnarray}\label{toy_model_1}
\left(-\frac{1}{2}\Delta +V \right)u = \lambda u,
\end{eqnarray}
where $\Omega=[-1,1]^2$ and 
$\displaystyle V({\vr})=-1/{|\vr|}$. 

\begin{figure}[!t]
  \centering 
\subfigure[The decomposition of $\Omega$]{\includegraphics[width=0.35\textwidth,height=0.35\textwidth]{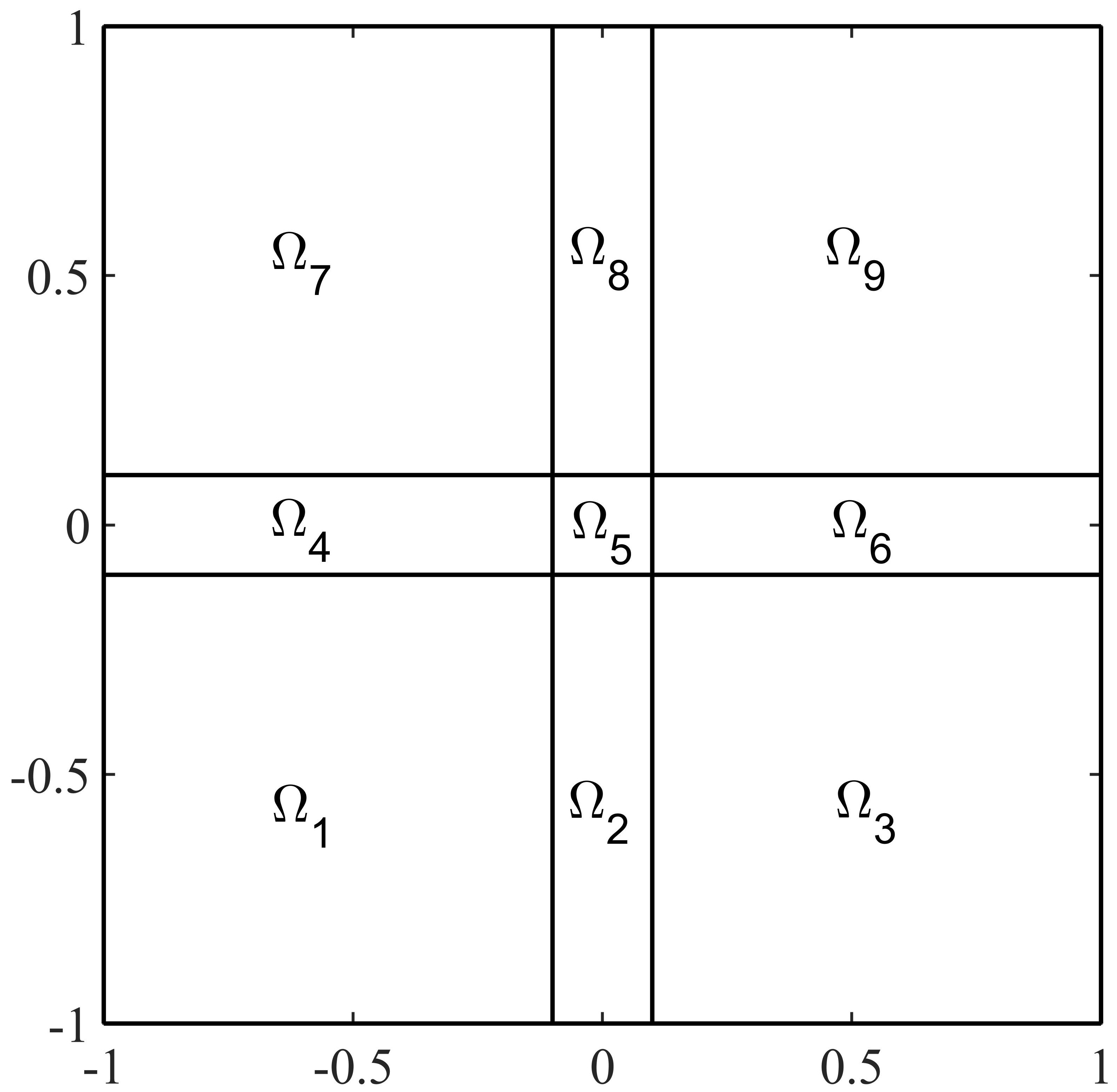}}
\label{The initial non-matching mesh.}
\subfigure[The initial non-matching mesh]{\includegraphics[width=0.35\textwidth,height=0.35\textwidth]{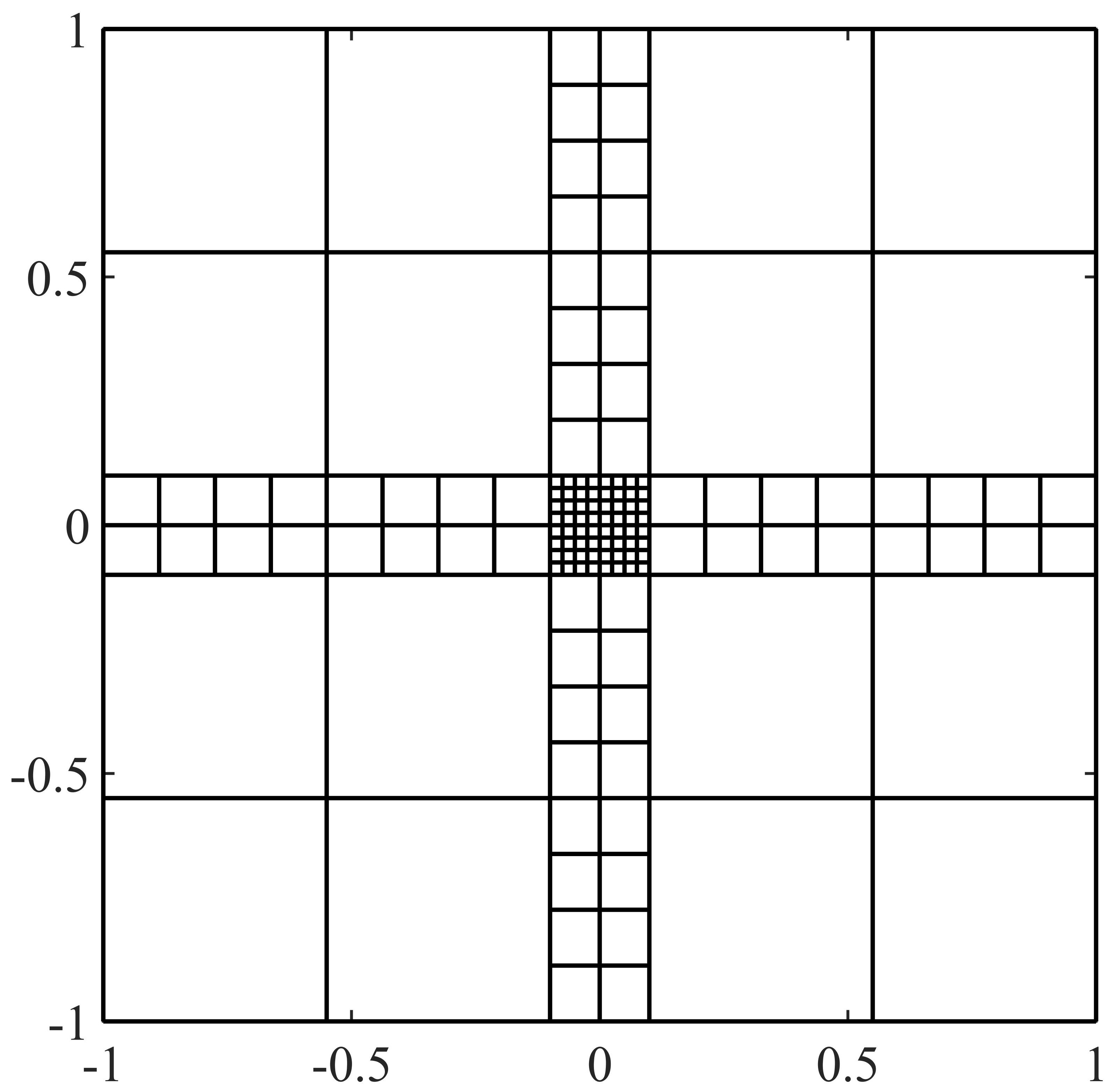}} \label{fig:Example_Mesh}
  \caption{The patch decomposition and initial mesh of computational domain. }\label{fig:Ex_1_2D}
\end{figure}

In the simulation, the domain $\Omega$ is divided into 9 patches, see Figure \ref{fig:Ex_1_2D}(a), and the initial mesh, see Figure \ref{fig:Ex_1_2D}(b), includes 144 elements, where $h_{\rm max}=0.45$ and $h_{\rm min} = 0.025 $. 
The reference values for the first four eigenvalues are
\[
\lambda_1 = -1.51501061, ~
\lambda_2 =  4.26228893, ~
\lambda_3 =  4.26228893, ~ 
\lambda_4 =  8.38293009,
\]
and the corresponding reference eigenfunctions, denoted by $u_i$ ($1\le i\le 4$), are displayed in Figure \ref{fig:example_1_ref_uh}. 
For the 2D linear eigenvalue problem with the singular potential, the first eigenfunction satisfies $u_1\in H^{2-\epsilon}(\Omega)$ for any $\epsilon>0$, see \cite{maday2019regularity}. 
It is observed from Figure \ref{fig:example_1_ref_uh} that the
DG-IGA method can accurately describe the cusp of $u_1$ located at the origin, while the last three eigenfunctions are much smoother.

\begin{figure}[h!]
\subfigure[$u_1$]{\includegraphics[width=0.22\textwidth,height=0.223\textwidth]{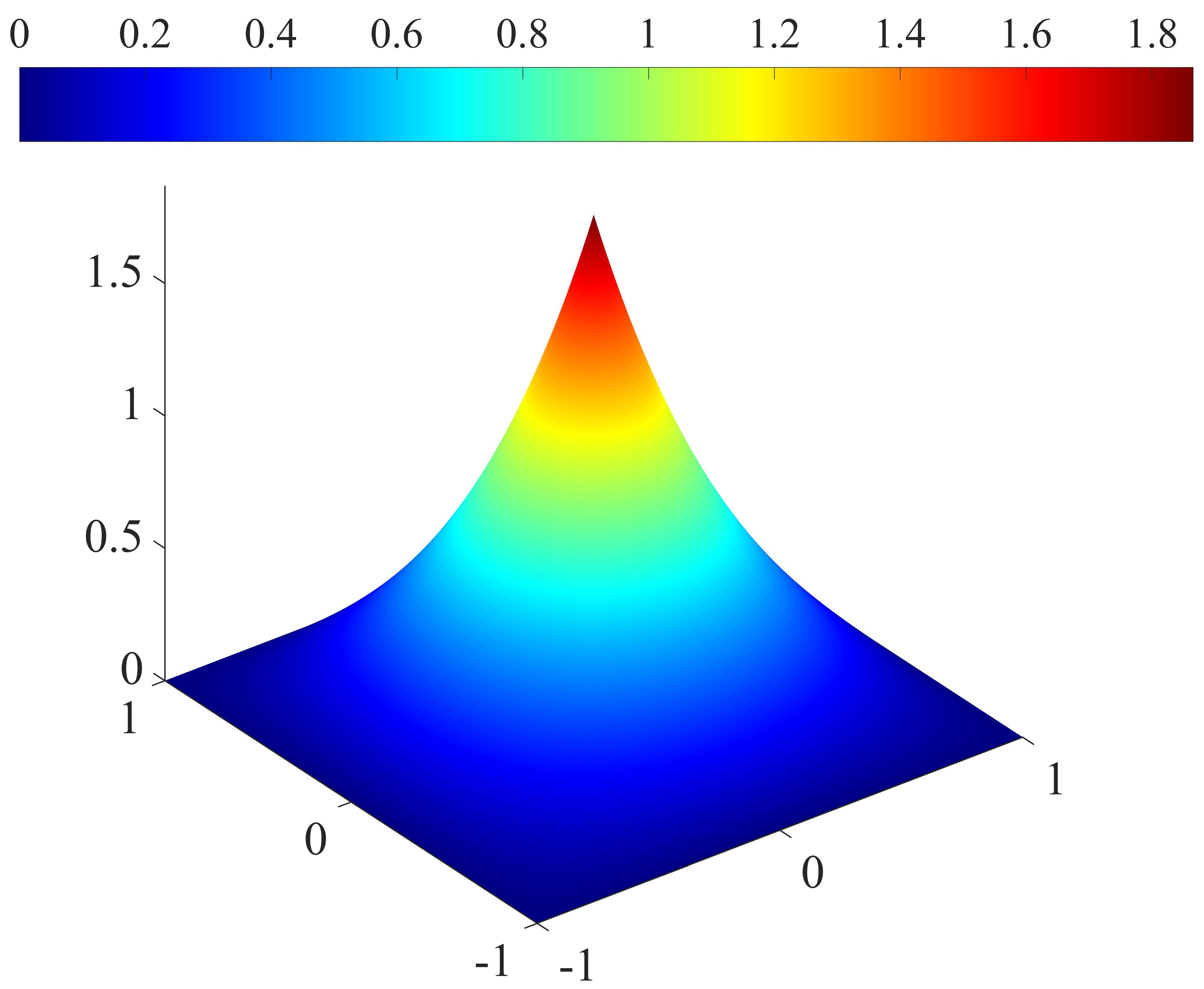}}
\subfigure[$u_2$]{\includegraphics[width=0.22\textwidth,height=0.223\textwidth]{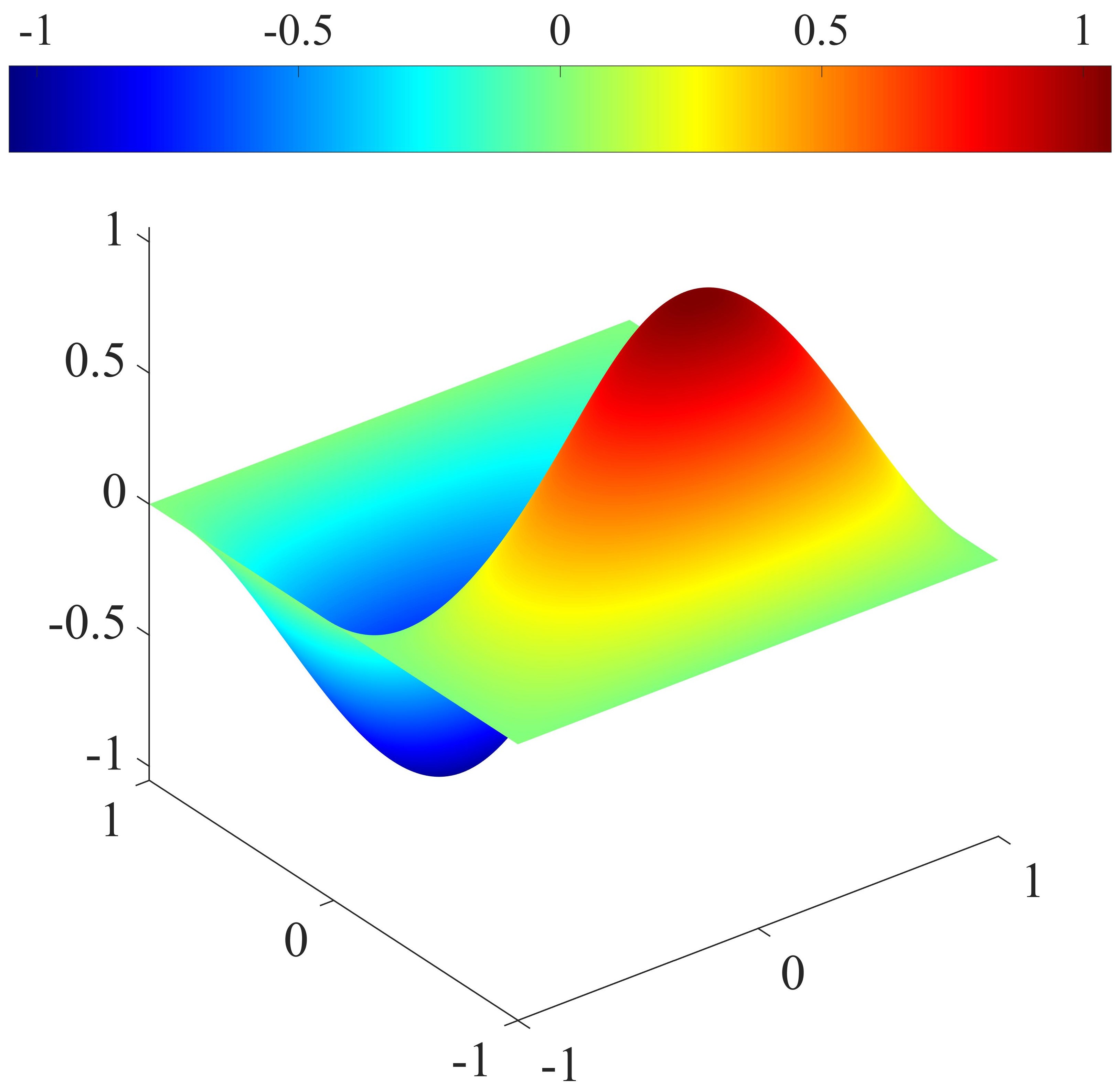}}
\subfigure[$u_3$]{\includegraphics[width=0.22\textwidth,height=0.223\textwidth]{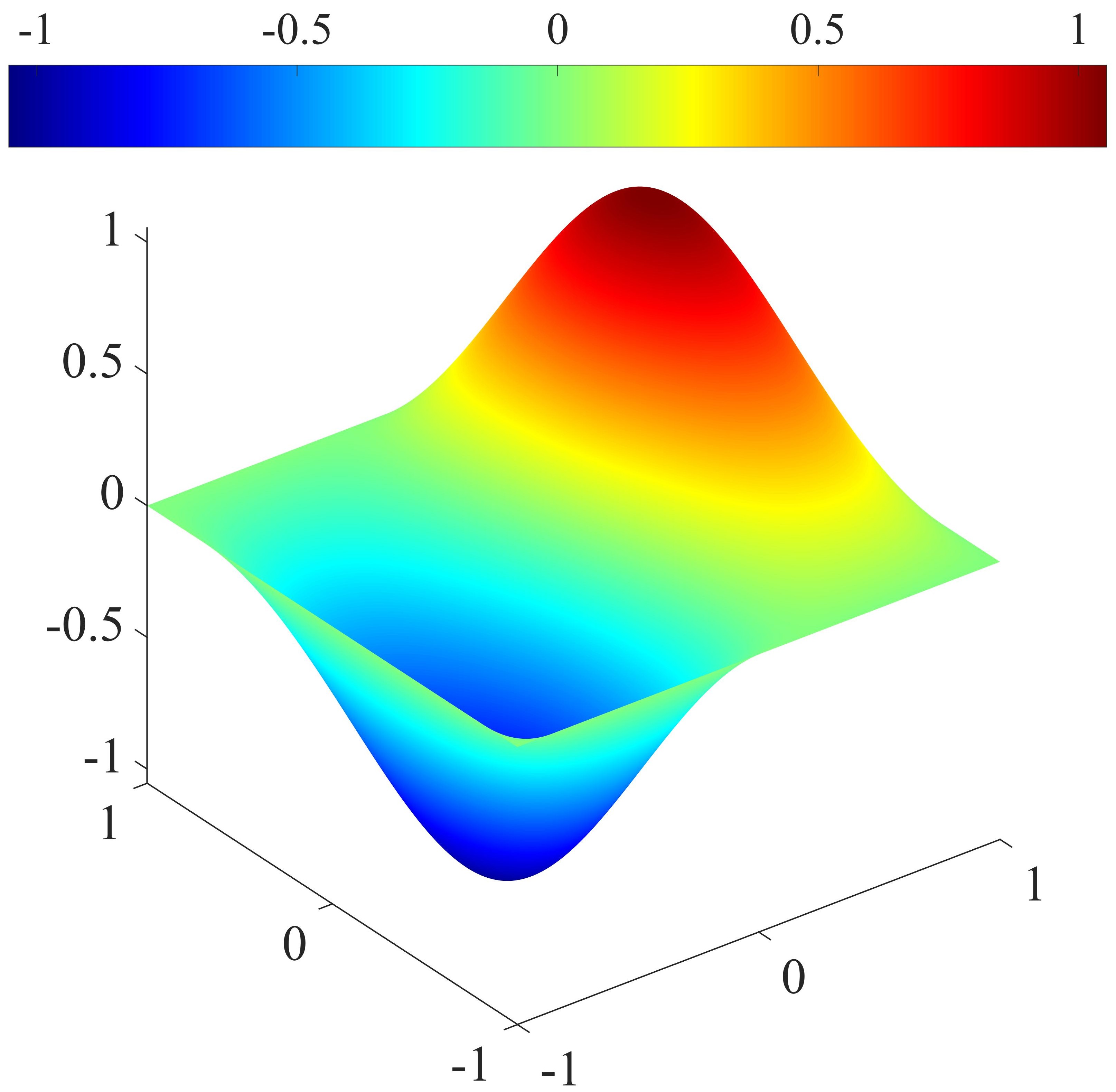}}
\subfigure[$u_4$]{\includegraphics[width=0.22\textwidth,height=0.223\textwidth]{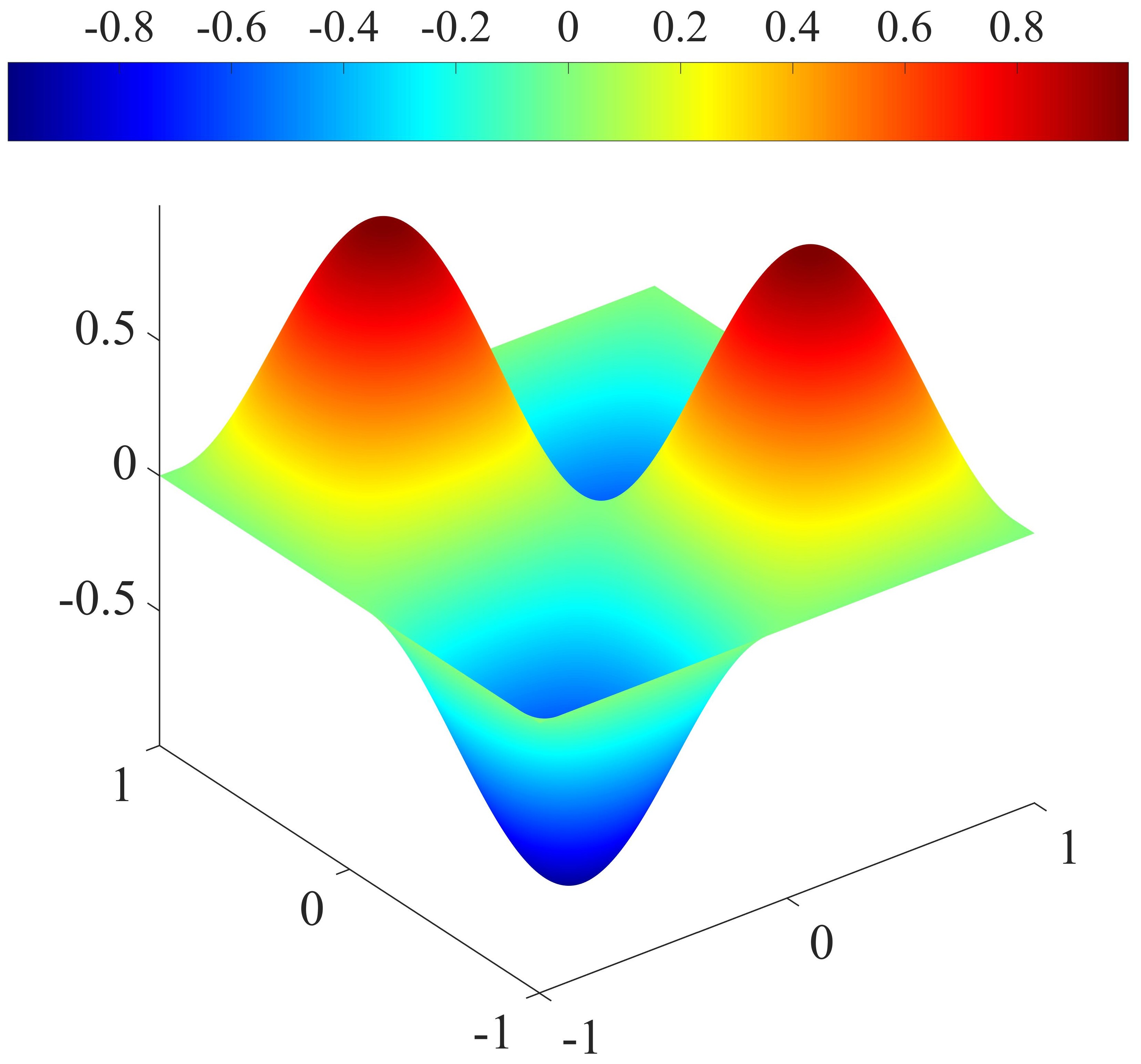}}
\caption{(Example 1) The first four reference eigenfunctions. } 
\label{fig:example_1_ref_uh}
\end{figure}

We present the first eigenvalue obtained by the DG-IGA and IGA methods in Table \ref{table:Example_1_1st_eigenvalue}. 
The results show that the DG-IGA method achieves higher accuracy with over 25 times fewer DOFs compared to the IGA method, demonstrating its superior performance when the corresponding eigenfunction has low regularity.

\renewcommand{\arraystretch}{1.5}  
\begin{table}[tp]  
  \centering 
  \scriptsize
    \begin{tabular}{c|c|c|c||c|c|c|c}   
    \hline    \multicolumn{4}{c||}{$p=1$} & \multicolumn{4}{c}{$p=2$}\cr 
    \hline    \multicolumn{2}{c|}{DG-IGA } & \multicolumn{2}{c||}{IGA} & \multicolumn{2}{c|}{DG-IGA } & \multicolumn{2}{c}{ IGA} \cr \hline
  DOFs  & $\lambda_1^{\rm DG}$  &  DOFs & $\lambda_1^{\rm IGA}$  &  DOFs  & $\lambda_1^{\rm DG}$  &   DOFs & $\lambda_1^{\rm IGA}$ \cr \hline   
  149,769  &   -1.51486965 & 4,198,401 &  -1.51469778 &
  152,100  &   -1.51499555 & 4,202,500 &  -1.51482259   \cr\hline 
    \end{tabular} 
\caption{(Example 1) The first eigenvalue computed by the DG-IGA  and IGA methods, where  $\lambda_1  = -1.51501061$.} \label{table:Example_1_1st_eigenvalue} 
\end{table}

The convergence of the first four eigenvalues with respect to $h_{\rm max}$ is presented in Figure \ref{fig:example_1_Eigenvalues_Error}, and the convergence of the first eigenfunction in $L^2$ and DG norms is shown in Figure \ref{fig:example_1_u1_h_Error}. It is observed that the convergence rates depend on both the degree of B-splines and the regularity of eigenfunctions. Moreover, the numerical convergence rates agree well with our theoretical results.
We also present the distribution of the absolute error for the first eigenfunction over $\Omega$ in Figure \ref{fig:example_1_EigenFunction_Error}. 
It can be found that the error becomes smaller as the initial mesh is successively refined, confirming the convergence of our method.

\begin{figure}[!t]
  \centering 
\subfigure[First eigenvalue]{\includegraphics[width=0.48\textwidth]{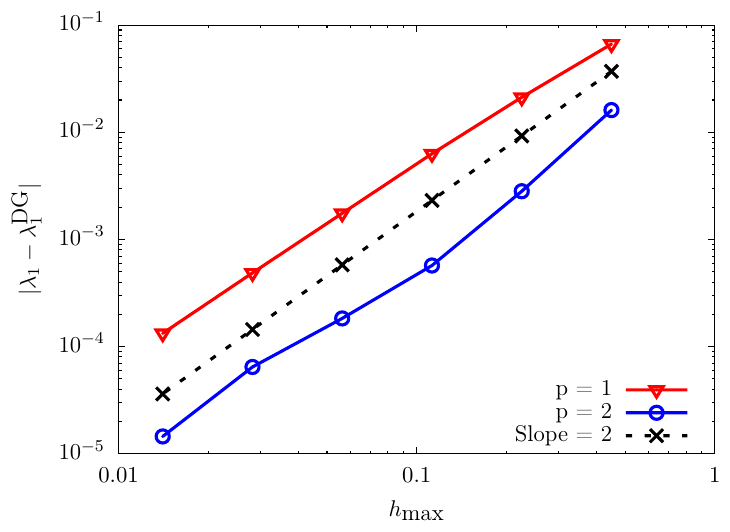}}
  \subfigure[Second-forth eigenvalues]{\includegraphics[width=0.48\textwidth]{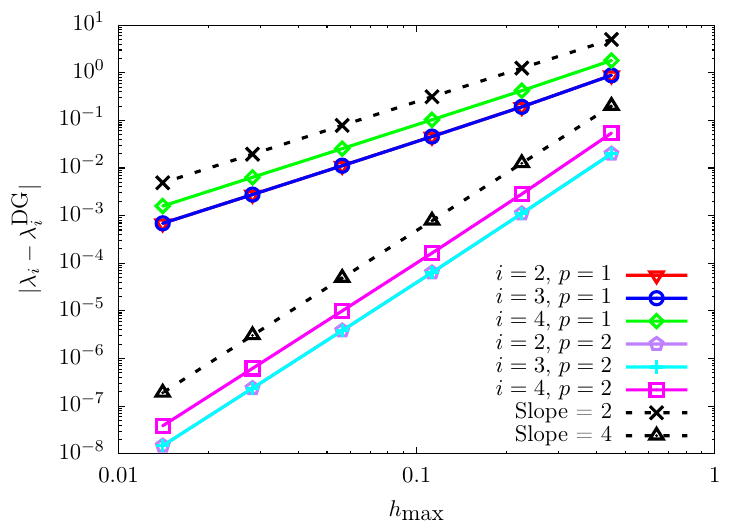}}  
  \caption{(Example 1) The convergence of the first four eigenvalues. } \label{fig:example_1_Eigenvalues_Error}
\end{figure}

\begin{figure}[!t]
  \centering 
  {\includegraphics[width=0.6\textwidth]{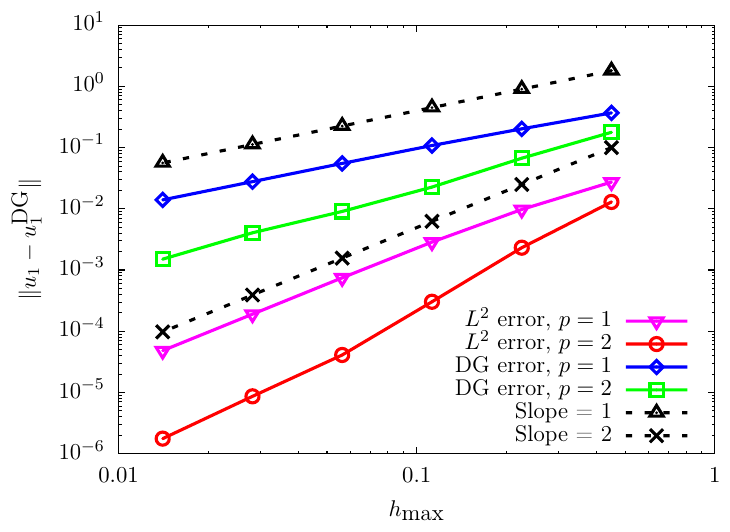}} 
  \caption{(Example 1) The convergence of the 
  first eigenfunction. 
  } \label{fig:example_1_u1_h_Error}
\end{figure}

\begin{figure}[!t]
  \centering 
\subfigure[$h_{\rm max}=0.45 $]{\includegraphics[width=0.31\textwidth,height=0.19\textwidth]{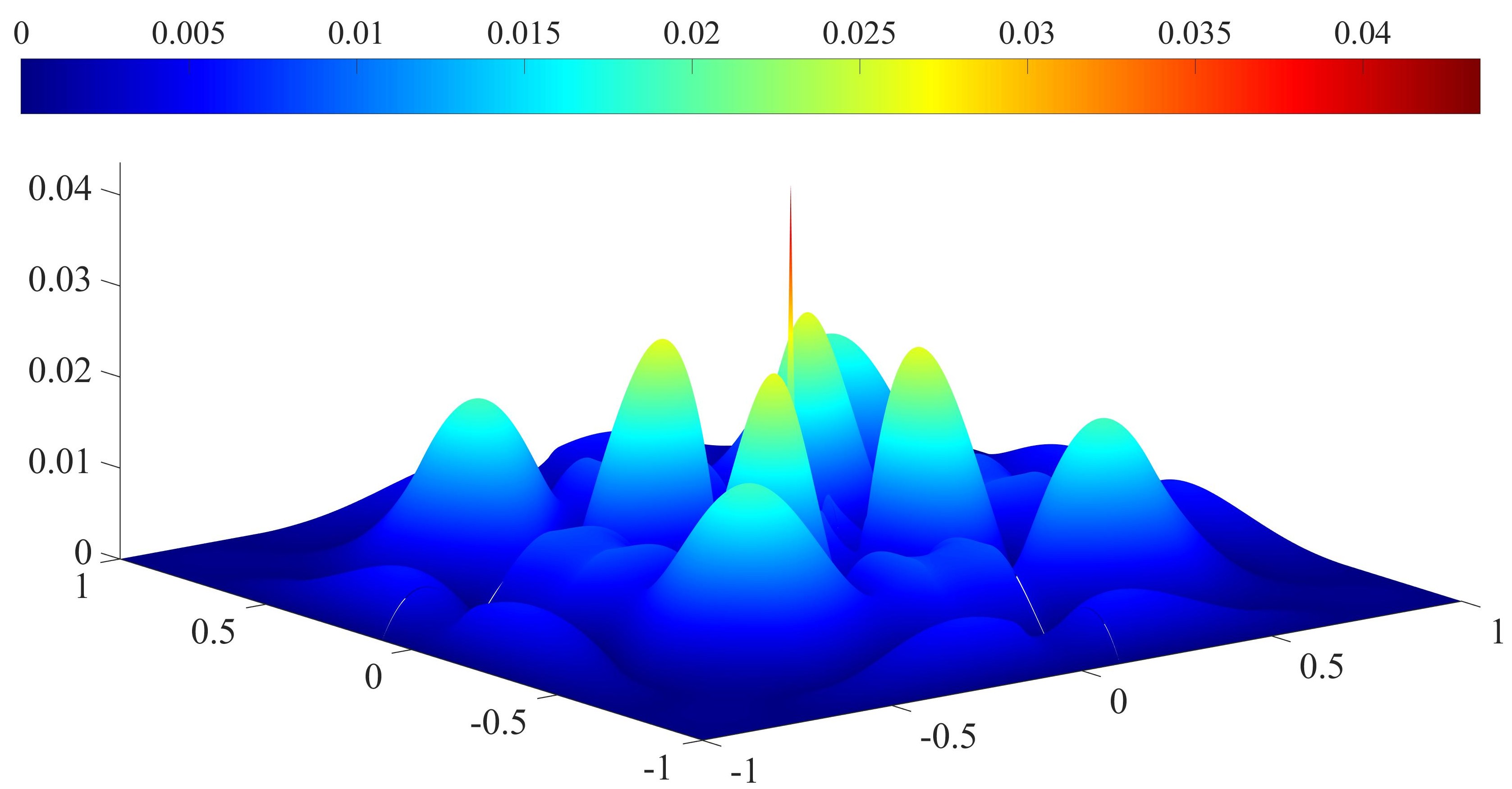}}
  \subfigure[$h_{\rm max}=0.45/2 $]{\includegraphics[width=0.31\textwidth,height=0.187\textwidth]{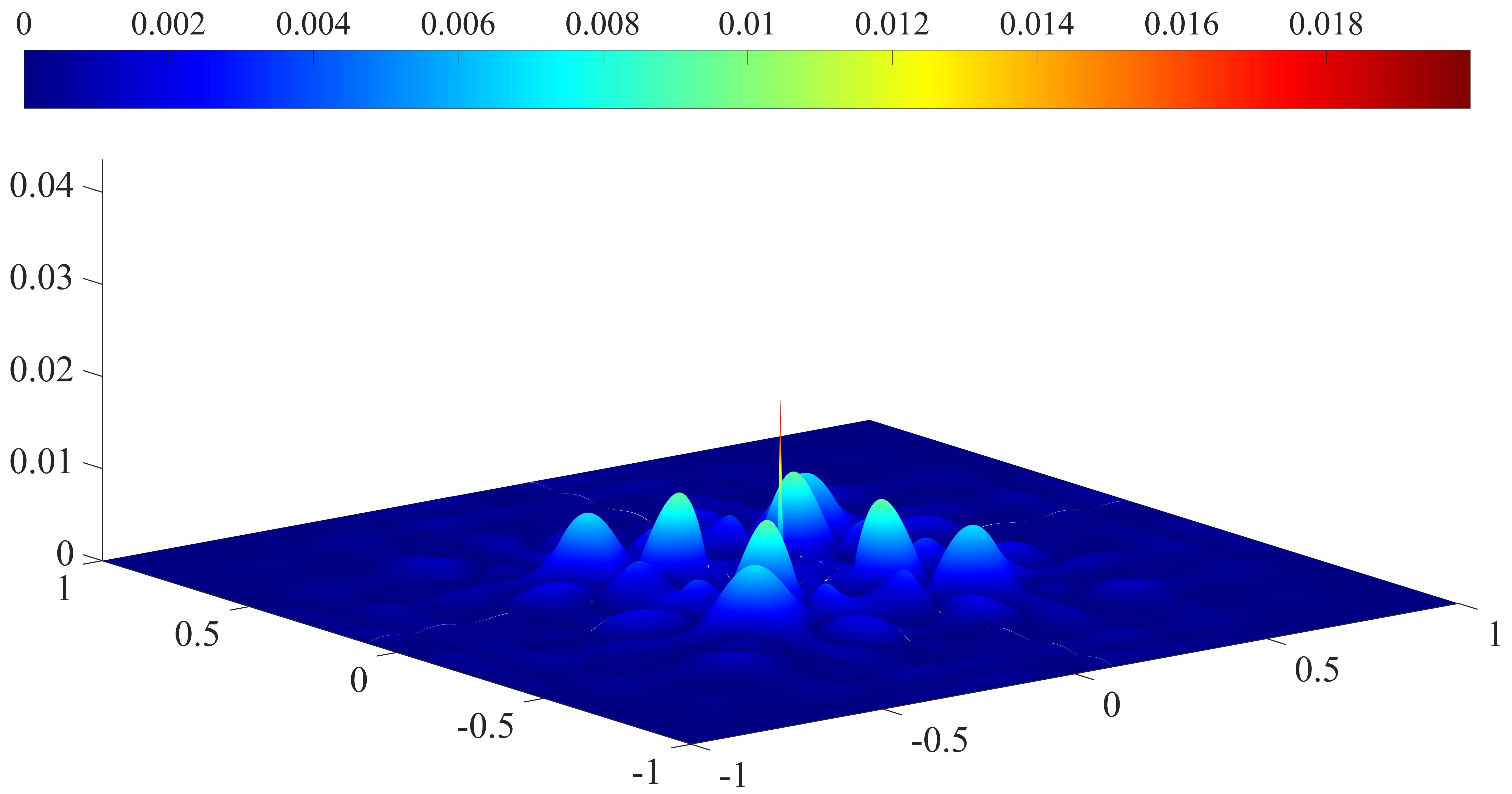}} 
    \subfigure[$h_{\rm max}=0.45/2^2$ ]{\includegraphics[width=0.31\textwidth,height=0.20\textwidth]{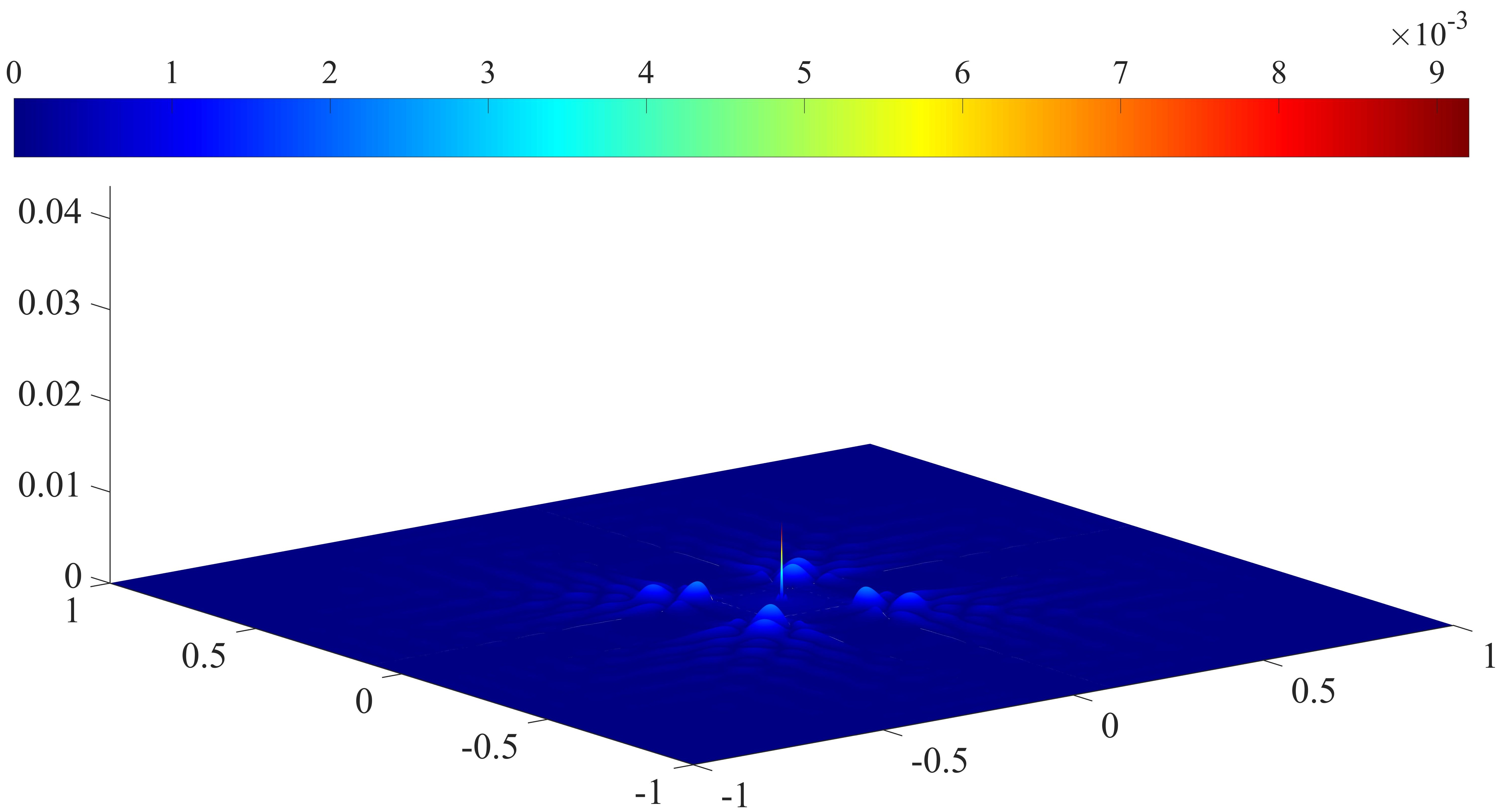}}  
  \caption{(Example 1) The distribution of $|u_1 - u_1^{\rm DG}|$ on three different meshes with $p=2$. } \label{fig:example_1_EigenFunction_Error}
\end{figure}

Finally, we test the effectiveness of the multiscale refinement strategy discussed in Remark \ref{remark:multiscale-refinement}. 
We show the convergence of first eigenpair in Figure \ref{fig:example_1_u1_h_Error_multiscale_refine}, where the relation $h_{\rm min} = \mathcal{O}(h_{\rm max}^2)$ is maintained during the mesh refinement, and it can be observed that the convergence rates of both eigenvalue and eigenfunction depend only on $p$. 
The results are different from those presented in Figure \ref{fig:example_1_Eigenvalues_Error}(a) and Figure \ref{fig:example_1_u1_h_Error}, where the convergence rates depend on both the degree $p$ and the regularity of $u_1$. These numerial convergence results show that the multiscale refinement strategy can significantly improve the convergence rate of numerical error for problems with non-smooth solutions.

\begin{figure}[!t]
  \centering 
  {\includegraphics[width=0.45\textwidth]{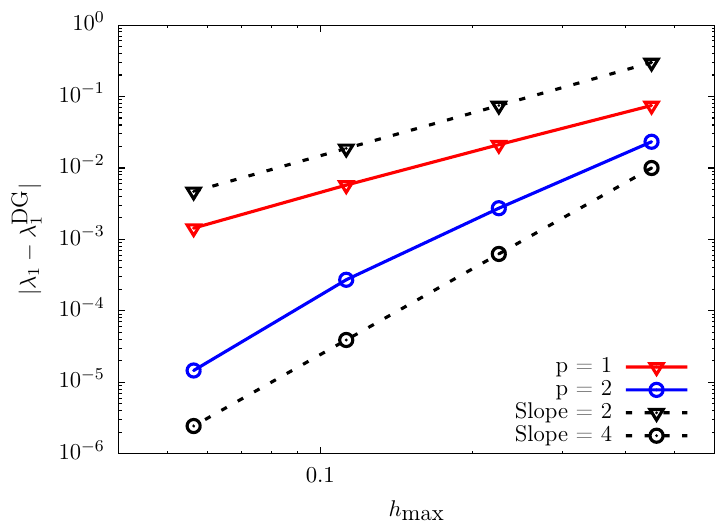}}
  {\includegraphics[width=0.45\textwidth]{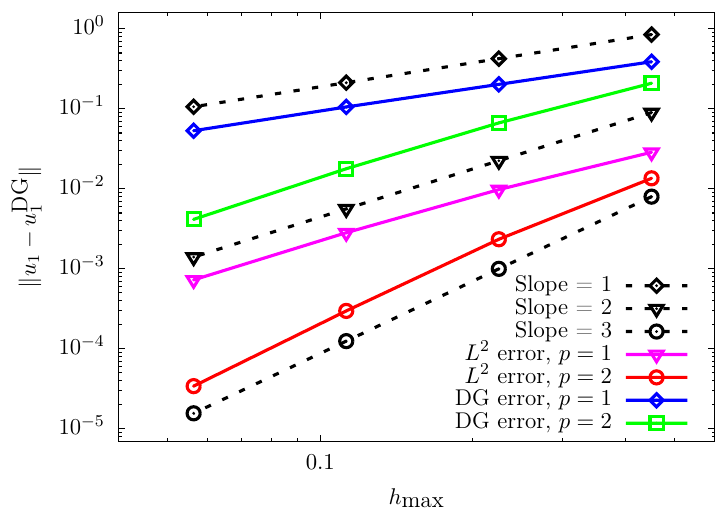}} 
  \caption{(Example 1) The convergence of the first eigenpair, where the multiscale refinement strategy is used.} \label{fig:example_1_u1_h_Error_multiscale_refine}
\end{figure}

\vspace{0.5cm}

\noindent{\bf Example 2 (2D linear problem for a two-atom system).}
Consider the two-dimensional linear eigenvalue problem for a two-atom system: Find $\lambda\in\mathbb{R}$ and $u\in H^1_0(\Omega)$ with $\|u\|_{L^2(\Omega)}=1$ such that
\begin{eqnarray}\label{toy_model_2}
\left(-\frac{1}{2}\Delta +V_{1}+V_{2} \right)u = \lambda u,
\end{eqnarray}
where $\Omega=[-1,1]^2$ and
$\displaystyle V_{j}(\vr) = -\frac{1}{|\vr-\vR_j|}~(j=1,2)$
with ${\bf R_1} = (-0.5,0)$ and ${\bf R_2} = (0.5,0)$ the positions of atoms. 

In this example, the domain $\Omega$ is divided into 15 patches, and the initial mesh contains 405 elements, where $h_{\rm max} =  0.45$ and $h_{\rm min } = 0.025$. 
The reference values for the first four eigenvalues are 
\[
\lambda_1  = -2.7618446474, ~ \lambda_2  =  -0.1848511166, ~ \lambda_3  =  3.0990197746, ~\lambda_4  =  6.7413028611.
\]
The  corresponding eigenfunctions are displayed in Figure \ref{fig:example_2_ref_uh}, from which we can observe that the regularity of the first two eigenfunctions is lower  than the other two eigenfunctions.

\begin{figure}[!t]
\subfigure[$u_1 $]{\includegraphics[width=0.24\textwidth,height=0.223\textwidth]{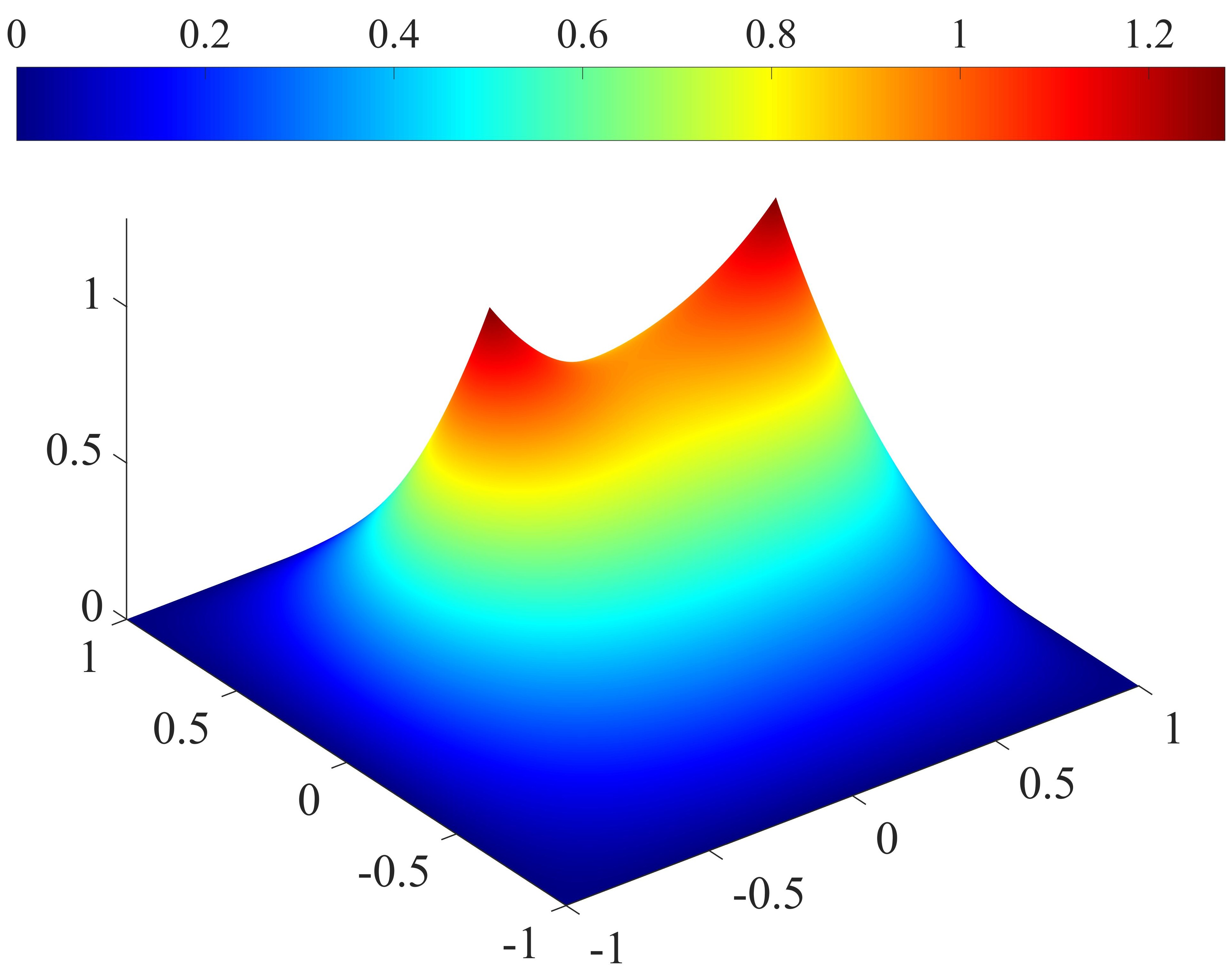}}
\subfigure[$u_2 $]{\includegraphics[width=0.24\textwidth,height=0.221\textwidth]{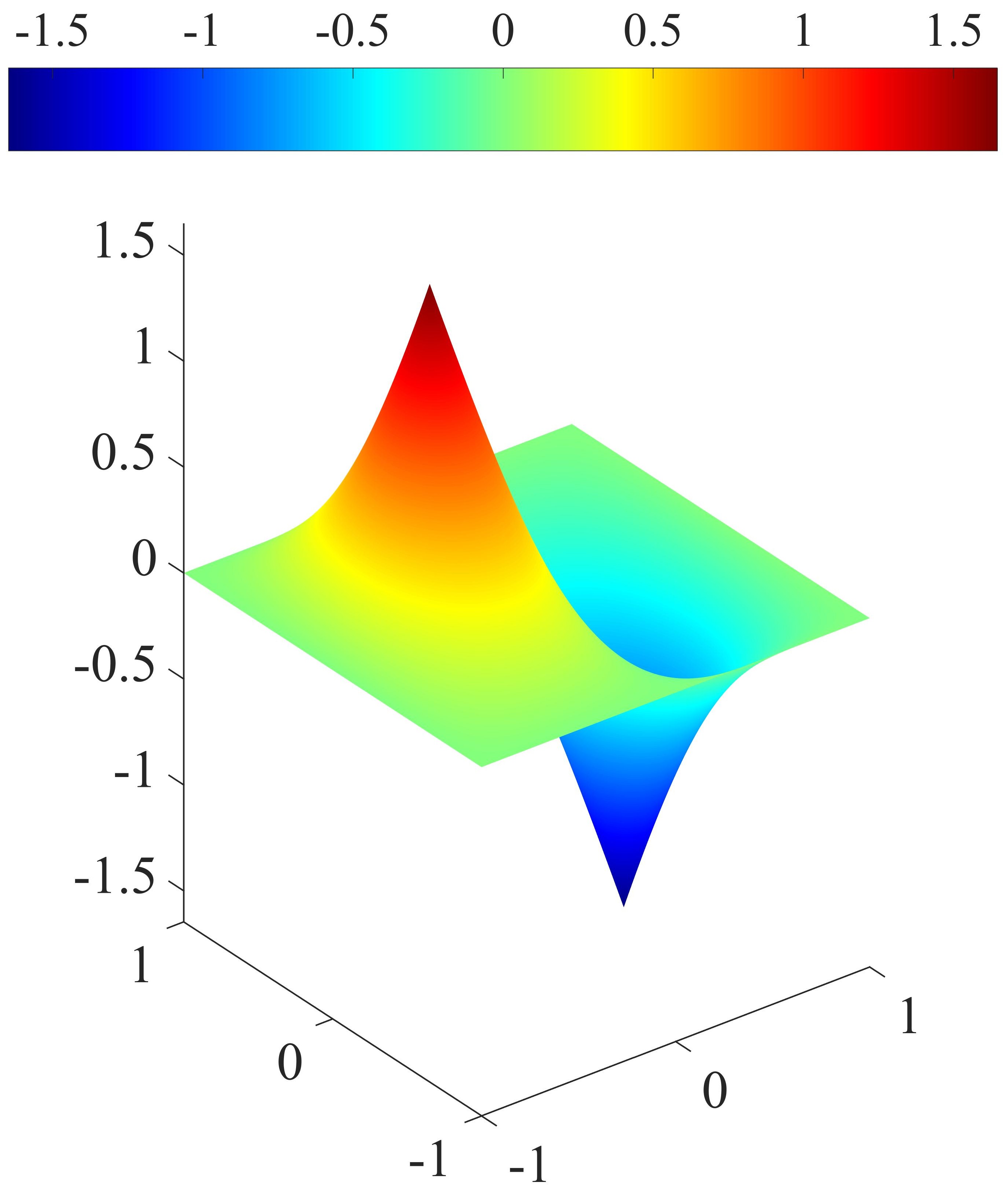}}
\subfigure[$u_3 $]{\includegraphics[width=0.24\textwidth,height=0.223\textwidth]{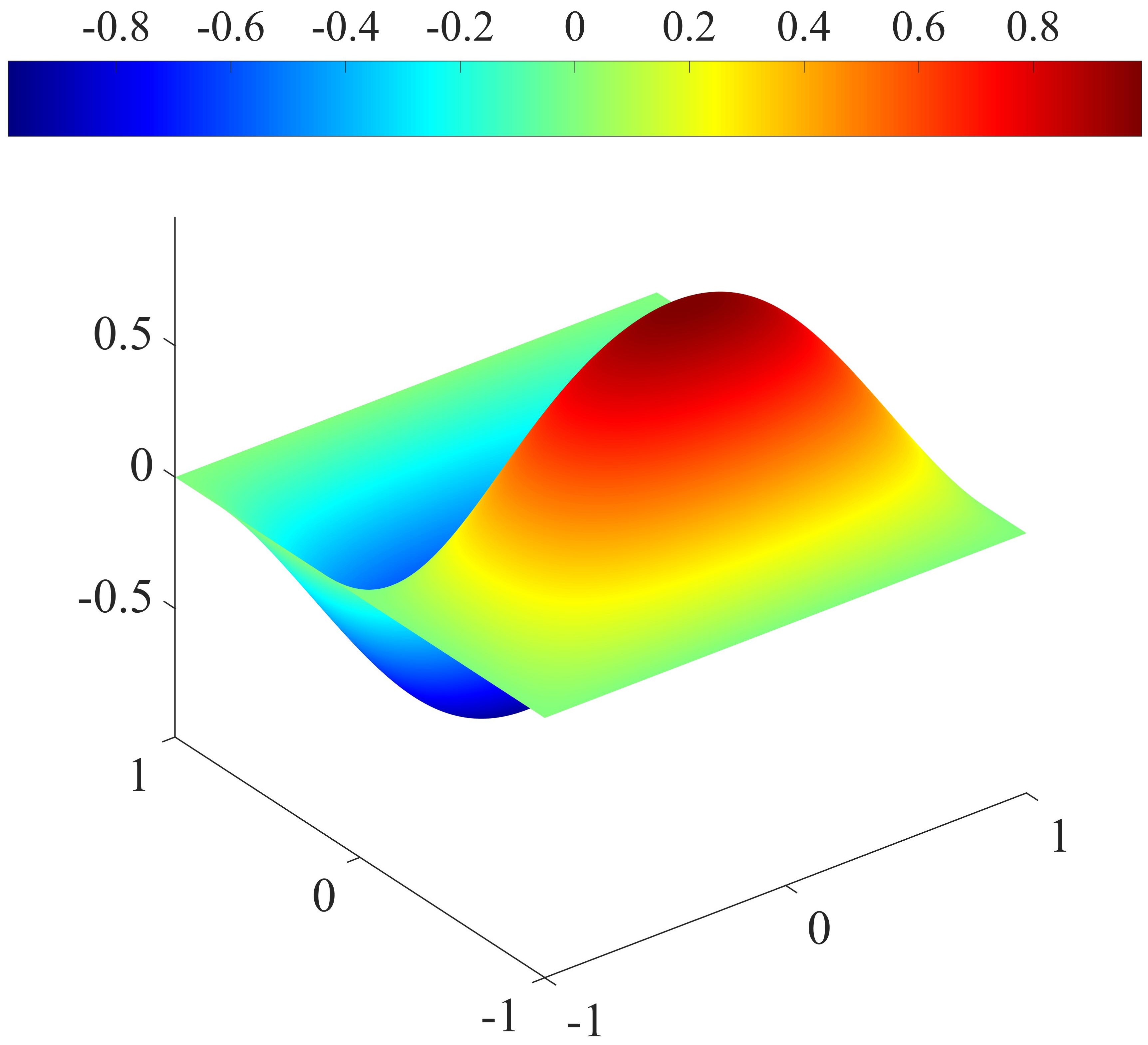}}
\subfigure[$u_4 $]{\includegraphics[width=0.24\textwidth,height=0.223\textwidth]{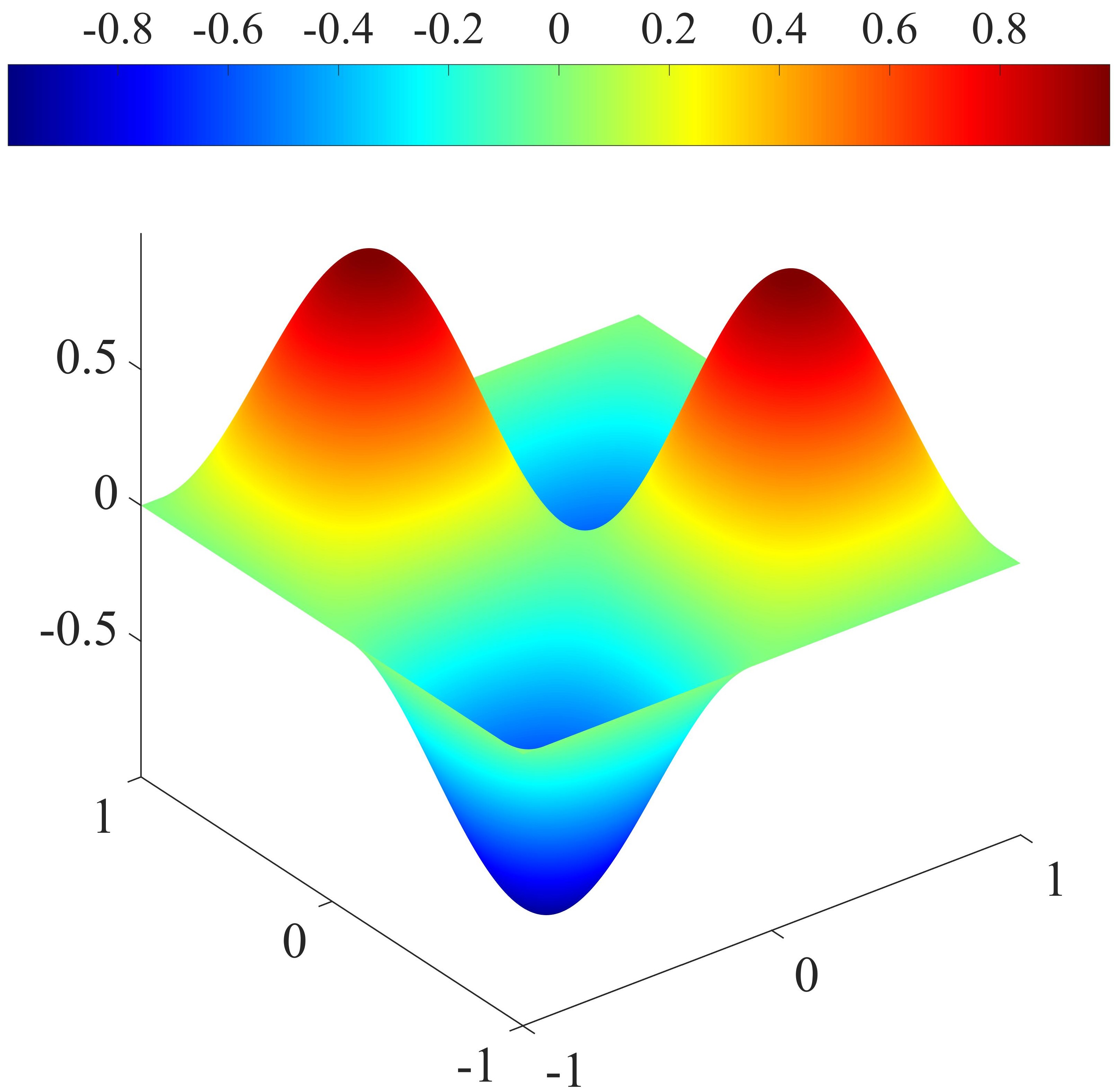}}
\caption{(Example 2) The first four reference eigenfunctions. 
} \label{fig:example_2_ref_uh}
\end{figure}

We show the first eigenvalue computed by the DG-IGA and IGA methods in Table \ref{table:Example_2_1st_eigenvalue}. From the table, we can draw the same conclusion as in Example 1, that is, the DG-IGA method outperforms the IGA method for the linear eigenvalue problem with singular potentials.

\renewcommand{\arraystretch}{1.5}  
\begin{table}[tp]  
  \centering 
  \scriptsize
    \begin{tabular}{c|c|c|c||c|c|c|c}   
    \hline    \multicolumn{4}{c||}{$p=1$} & \multicolumn{4}{c}{$p=2$}\cr 
    \hline    \multicolumn{2}{c|}{DG-IGA } & \multicolumn{2}{c||}{IGA} & \multicolumn{2}{c|}{DG-IGA } & \multicolumn{2}{c}{IGA} \cr \hline
  DOFs  & $\lambda_1^{\rm DG}$  &  DOFs & $\lambda_1^h$  &   DOFs  & $\lambda_1^{\rm DG}$  &   DOFs & $\lambda_1^h$ \cr \hline   
  274,383   &   -2.7617408370 & 
  4,198,401 &   -2.7615573535 &
  71,676    &   -2.7617839821 & 4,202,500 &   -2.7616748981  \cr\hline 
    \end{tabular} 
\caption{(Example 2) The first eigenvalue computed by the DG-IGA and IGA methods, where $\lambda_1  = -2.7618446474$.} \label{table:Example_2_1st_eigenvalue} 
\end{table} 

The convergence of the first four eigenvalues is presented in Figure \ref{fig:example_2_Eigenvalues_Error}, and the
convergence of the corresponding eigenfunctions in $L^2$ and DG norms is shown in Figure \ref{fig:example_2_EigenFunction_Error_p_1_2}. Similar to Example 1, it is observed that the convergence rate depends on both the degree of B-splines  and the regularity of eigenfunctions, and the numerical convergence results again confirm our theoretical result.

\begin{figure}[!t]
  \centering 
\subfigure[$p = 1$]{\includegraphics[width=0.48\textwidth]{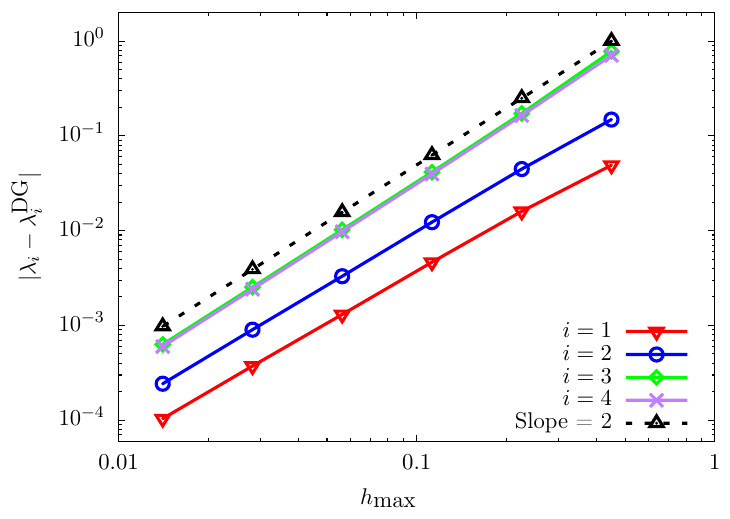}}
  \subfigure[$p = 2$]{\includegraphics[width=0.48\textwidth]{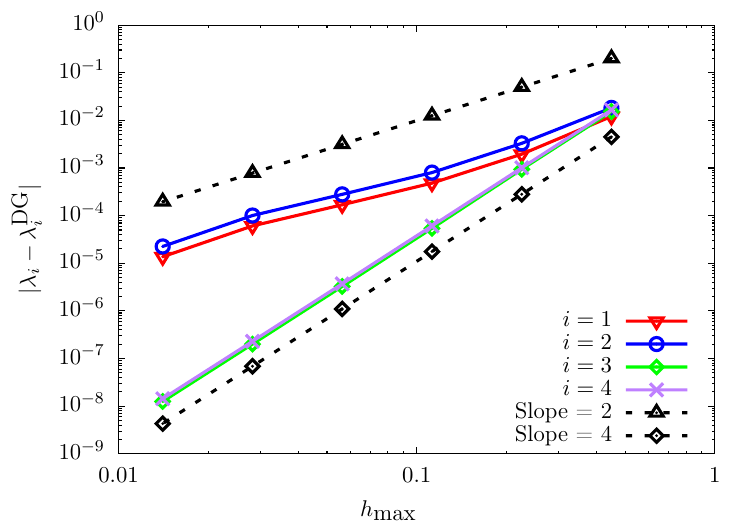}} 
\caption{(Example 2) The convergence of the first four eigenvalues.}
\label{fig:example_2_Eigenvalues_Error}
\end{figure}

\begin{figure}[h!]
  \centering 
\subfigure[$L^2$ norm error, $p=1$]{\includegraphics[width=0.45\textwidth]{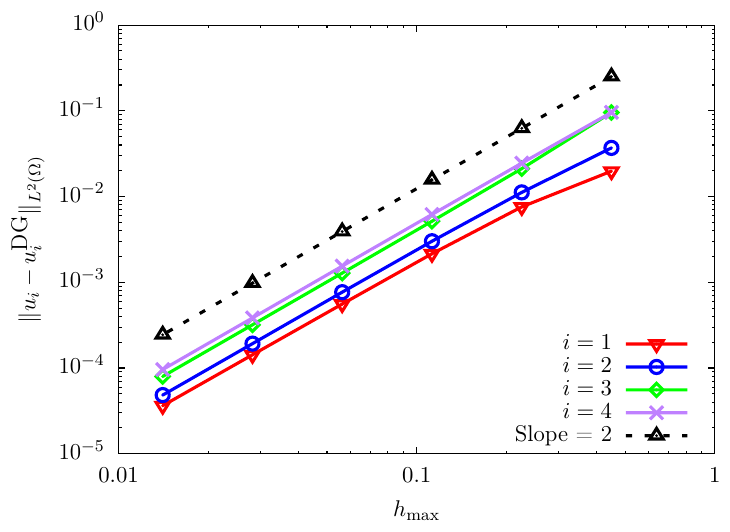}}
  \subfigure[DG norm error, $p=1$]{\includegraphics[width=0.45\textwidth]{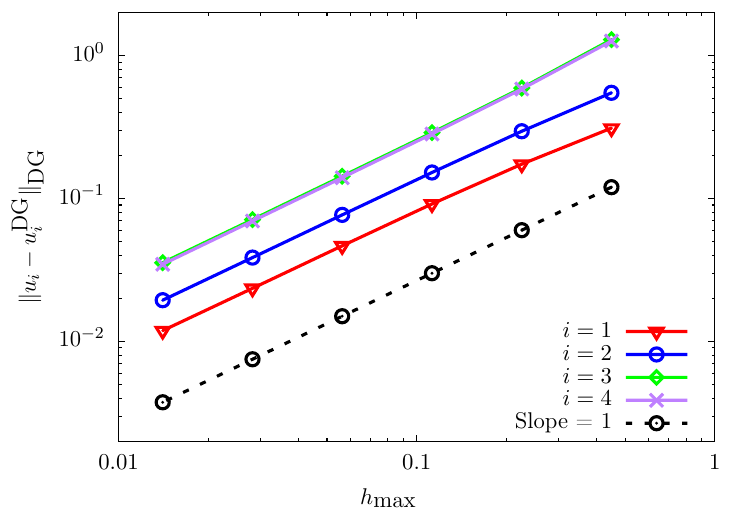}} 
 \subfigure[$L^2$ norm error,  $p=2$]{\includegraphics[width=0.45\textwidth]{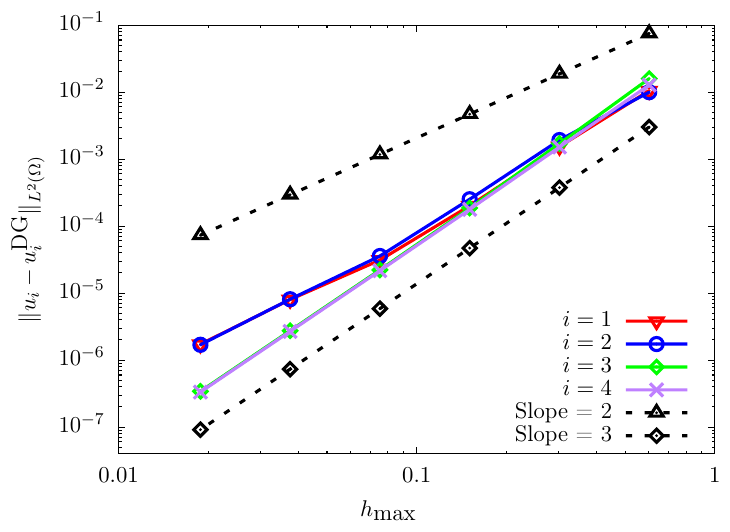}}
  \subfigure[DG norm error, $p=2$]{\includegraphics[width=0.45\textwidth]{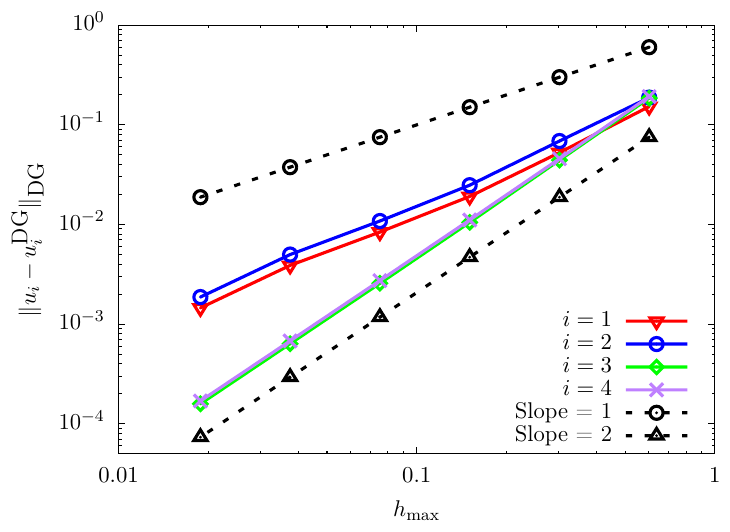}} 
\caption{(Example 2) The convergence of the first four eigenfunctions in $L^2$ and DG norms.}
\label{fig:example_2_EigenFunction_Error_p_1_2}
\end{figure}

\vspace{0.5cm}
\noindent{\bf Example 3 (modeling 2D Bose-Einstein condensates).}
Consider the following Gross-Pitaevskii equation originating from modeling the Bose-Einstein condensates \cite{pitaevskii2003bose}: Find $\lambda\in\mathbb{R}$ and $u\in H^1_0(\Omega)$ with $\|u\|_{L^2(\Omega)} = 1$ such that
\begin{eqnarray}\label{toy_model_3}
\left(-\frac{1}{2}\Delta + V + \rho \right)u= \lambda u,
\end{eqnarray}
where $\Omega=[-1,1]^2$, $\displaystyle V(\vr)=-1/|\vr|$, and $\rho(\vr)=|u(\vr)|^2$. For this nonlinear eigenvalue problem, the self-consistent field (SCF) iteration method \cite{kohn1965self} is utilized to linearize and iteratively solve the equation \eqref{toy_model_3}. The initial mesh for this example is the same as the one for Example 1, and the reference value for the first eigenvalue is $\lambda_1  = -0.5773370795 $. 

We present the convergence of the first eigenvalue and the corresponding eigenfunction in $L^2$ and DG norms in Figure \ref{fig:example_3_Eigenvalues_Error}. Although our theoretical analysis applies to the linear eigenvalue problem, very similar convergence results are observed for the nonlinear problem.

\begin{figure}[!t]
  \centering 
  {\includegraphics[width=0.45\textwidth]{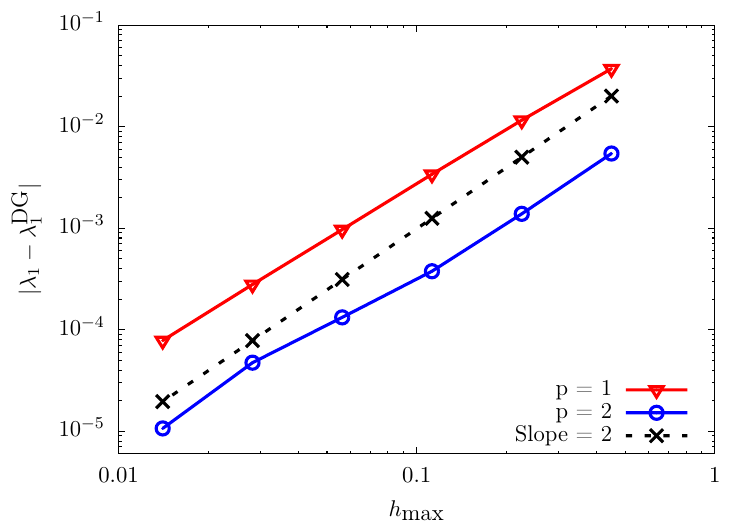}}
  {\includegraphics[width=0.45\textwidth]{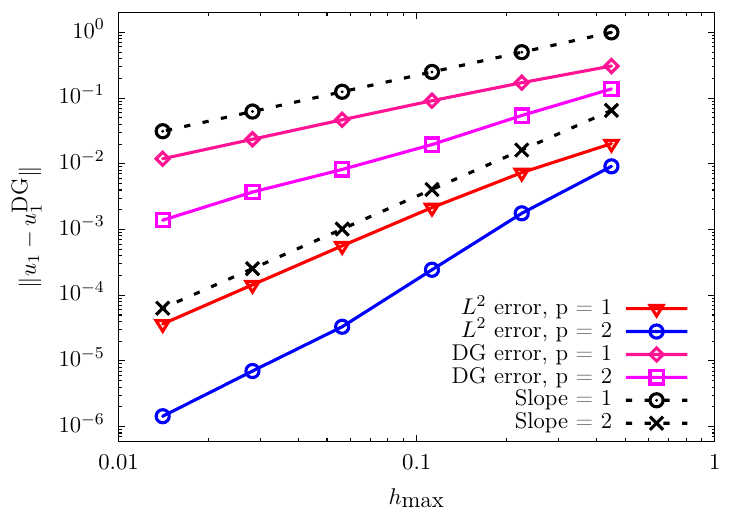}} 
  \caption{(Example 3) The convergence of the first eigenpair.  } 
  \label{fig:example_3_Eigenvalues_Error}
\end{figure}

\vspace{0.5cm}
\noindent{\bf Example 4 (3D linear problem for a single-atom system).}
Consider the three-dimensional linear eigenvalue problem: Find $\lambda\in\mathbb{R}$ and $u\in H^1_0(\Omega)$ with $\|u\|_{L^2(\Omega)} = 1$ such that
\begin{eqnarray}
\left(-\frac{1}{2}\Delta +V \right)u = \lambda u,
\end{eqnarray}
where $\Omega=[-2,2]^3$ and 
$\displaystyle V({\vr})=-1/|\vr|$. In this example, the patch that includes the origin is $[-0.5,0.5]^3$, and for the initial mesh, we set $h_{\rm max} = 1.5 $ and $h_{\rm min} = 0.25$. 
The reference value for the first eigenvalue is $\lambda_1  =  -0.2699102$, and the corresponding eigenfunction along the $x$-axis and the contour plot of $u_1 $ on the plane $z=0$ are shown in Figure \ref{fig:example_4_ref_uh}, from which it is found that the the DG-IGA method is able to accurately describe the cusp at the origin. For this three-dimensional linear eigenvalue problem, we have $u_1\in H^{5/2 - \epsilon}(\Omega)$ for any $\epsilon>0$, see \cite{maday2019regularity}. 

\begin{figure}[!t]
\centering 
\subfigure[$u_1$ along the $x$-axis]
{\includegraphics[width=0.45\textwidth,height=0.37\textwidth]{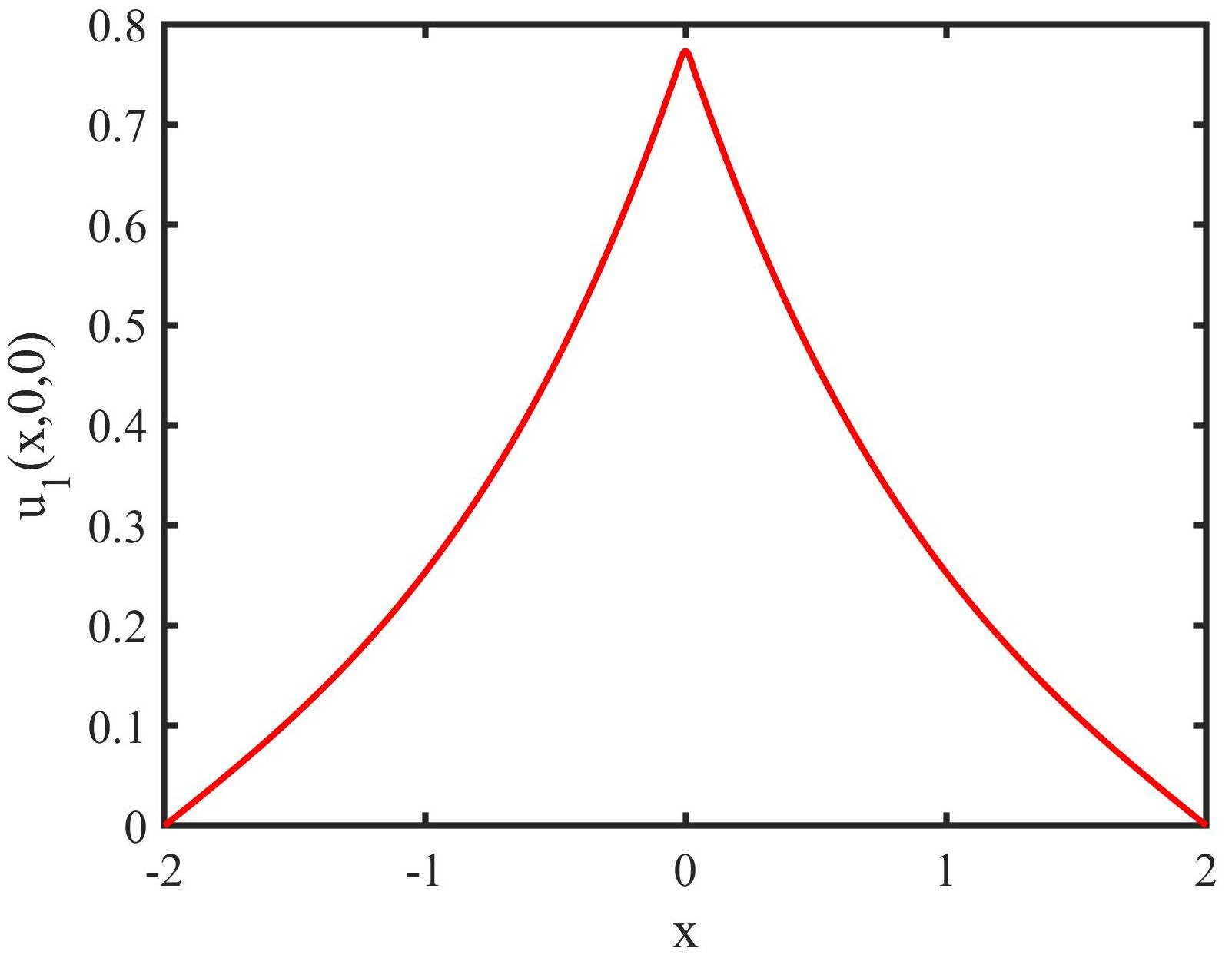}}
\hspace{0.2cm}
\subfigure[$u_1$ on the plane $z=0$]
{\includegraphics[width=0.45\textwidth]{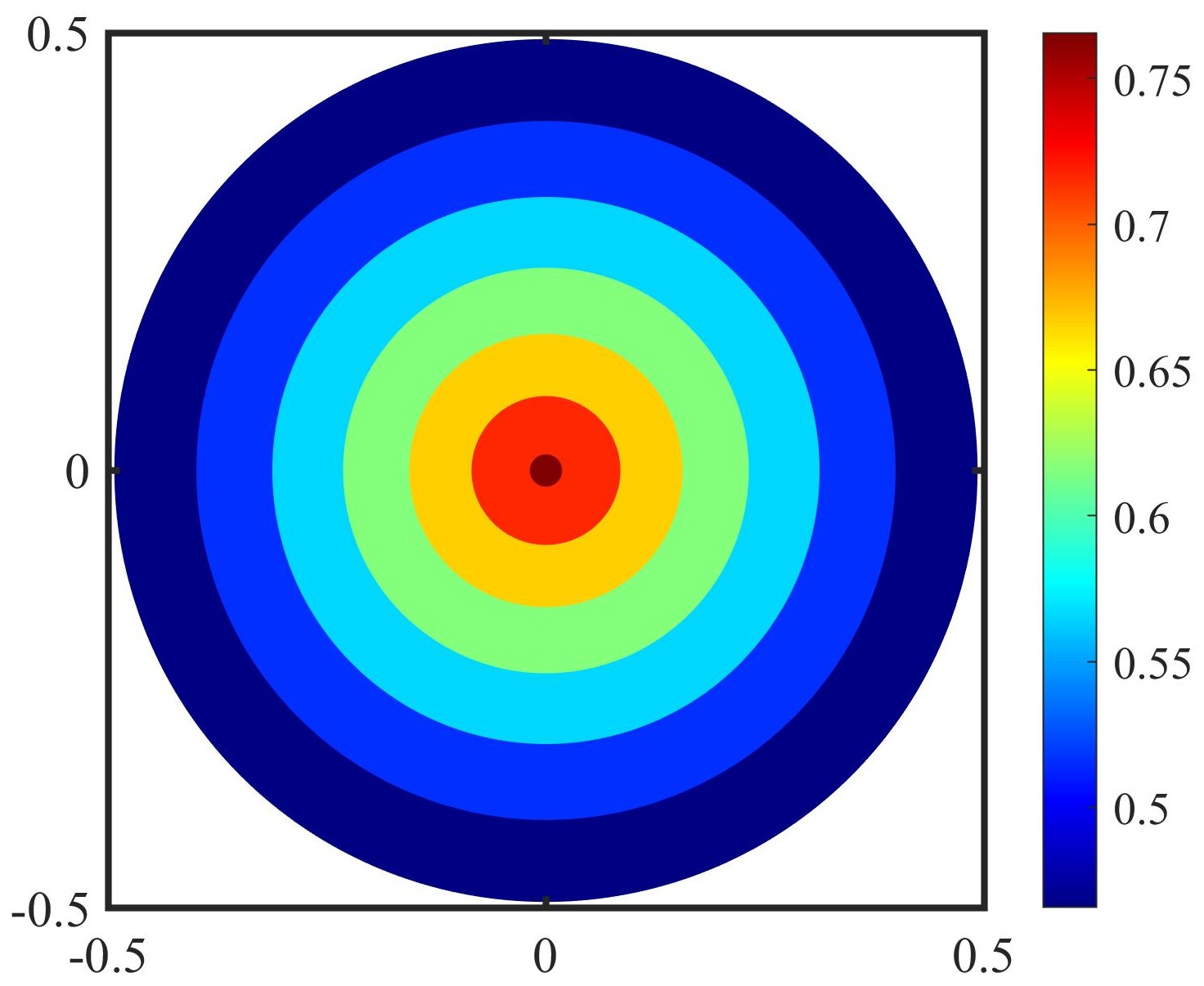}}
\caption{(Example 4)  (a) The eigenfunction $u_1 $ along the $x$-axis, and (b) the contour plot of $u_1$ on the plane $z=0$ with  $(x,y)\in [-0.5,0.5]^2$.} 
\label{fig:example_4_ref_uh}
\end{figure}

We present the convergence of the first eigenvalue and the corresponding eigenfunction in $L^2$ and DG norms in Figure \ref{fig:example_4_EigenFunction_Error}. Similar to the 2D linear eigenvalue problem, it is observed that the convergence rate depends on both the regularity of eigenfuntion and the degree of B-splines, which agrees well with the theoretical convergence result.

\begin{figure}[!t]
  \centering 
  {\includegraphics[width=0.45\textwidth]{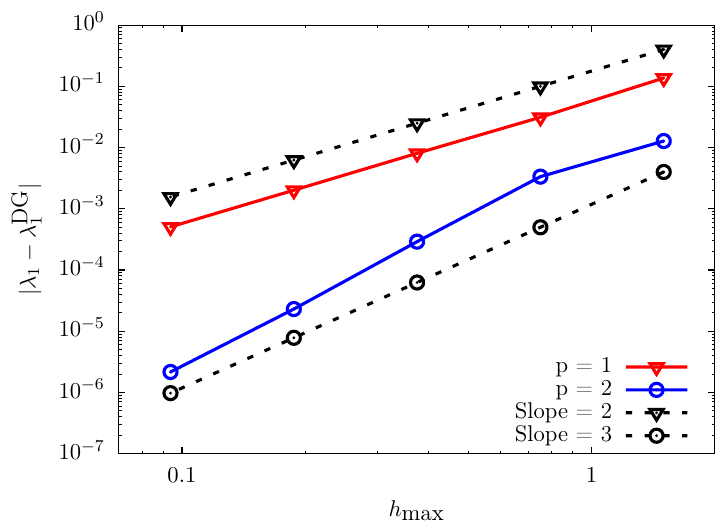}}
  {\includegraphics[width=0.45\textwidth]{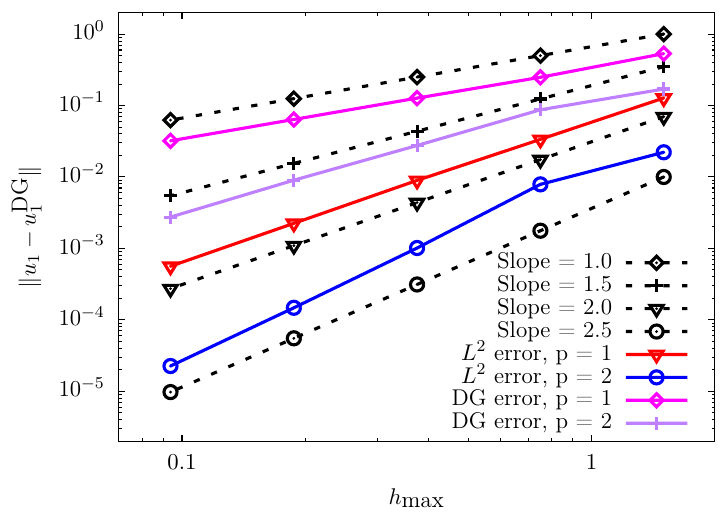}} 
  \caption{(Example 4) The convergence of the  first eigenpair. } \label{fig:example_4_EigenFunction_Error}
\end{figure}

\vspace{0.5cm}
\noindent{\bf Example 5 (simulation of a helium atom).}
In this example, we consider the simulation of a helium atom: Find $\lambda\in\mathbb{R}$ and $u\in H^1_0(\Omega)$ with $\|u\|_{L^2(\Omega)}=1$ such that
\begin{eqnarray}\label{toy_model_4}
\left(-\frac{1}{2}\Delta + V_{\rm ext} + V_{\rm H} [\rho] + V_{\rm xc}[\rho] \right)u = \lambda u,
\end{eqnarray}
where $\Omega=[-10,10]^3$, the external potential  $\displaystyle V_{\rm ext}({\bf r})=-\frac{2}{|\vr|}$, the Hartree potential 
$\displaystyle V_{\rm H} [\rho]=\int_{\mathbb{R}^3}\frac{\rho({\bf \vr'})}{|\cdot-{\vr'}|} \dd{\bf \vr'}$,  the electron density $\rho({\vr})=2|u(\vr)|^2$, and the exchange-correlation potential $V_{\rm xc}[\rho]$ is generated by the local density
approximation (LDA) \cite{MARQUES2012}. In the simulation, the patch that contains the origin is $[-1,1]^3$, and we set $h_{\rm max} = 4.5$ and $h_{\rm min} = 0.5$ in the initial mesh. 
In this example, we compute the ground state energy of the helium system (see, e.g., \cite{chen13} for a detailed expression of the total energy), and the reference ground state energy, $E^{\rm ref}  = -2.83428$ a.u., was obtained by the software \emph{NWChem} \cite{VALIEV20101477} with aug-cc-pv6z basis set. 

We present the distribution of electron density $\rho$ along the $x$-axis in Figure \ref{fig:example_5_uh_1_p_2}(a), and display the contour plot of $\rho$ on the plane $z=0$  with $(x,y)\in [-1,1]^2$ in Figure \ref{fig:example_5_uh_1_p_2}(b), where the numerical solution is obtained by the DG-IGA method with $h_{\rm max}=4.5/2^3$ and $p=2$. Similar to the 3D linear eigenvalue problem, it can be observed that the cusp at the origin has been accurately described.

\begin{figure}[!t]
  \centering 
  \subfigure[$\rho$ along the $x$-axis]
  {\includegraphics[width=0.435\textwidth]{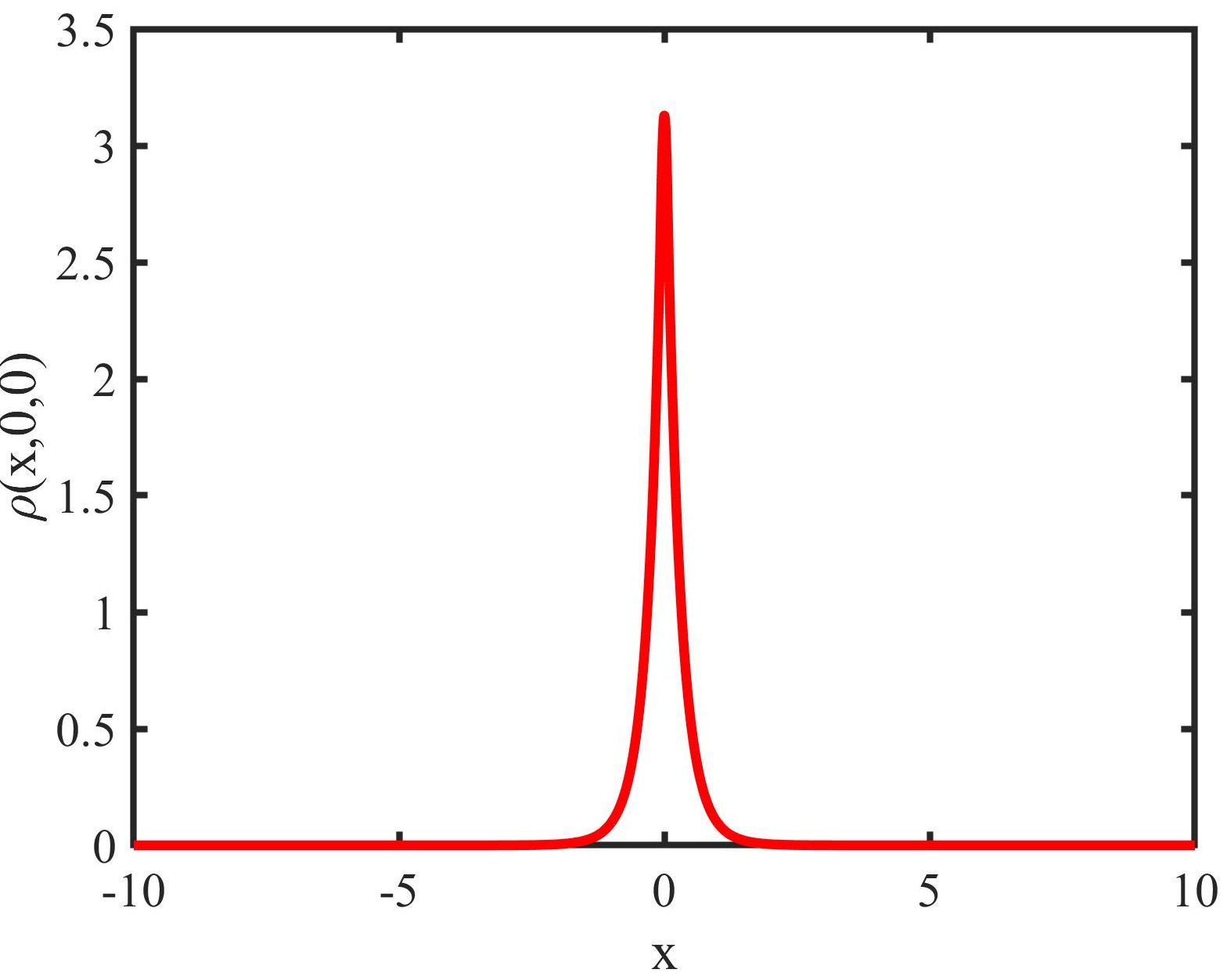}}
  \hspace{0.2cm}
  \subfigure[$\rho$ on the plane $z=0$]
{ \includegraphics[width=0.42\textwidth,height=0.35\textwidth]{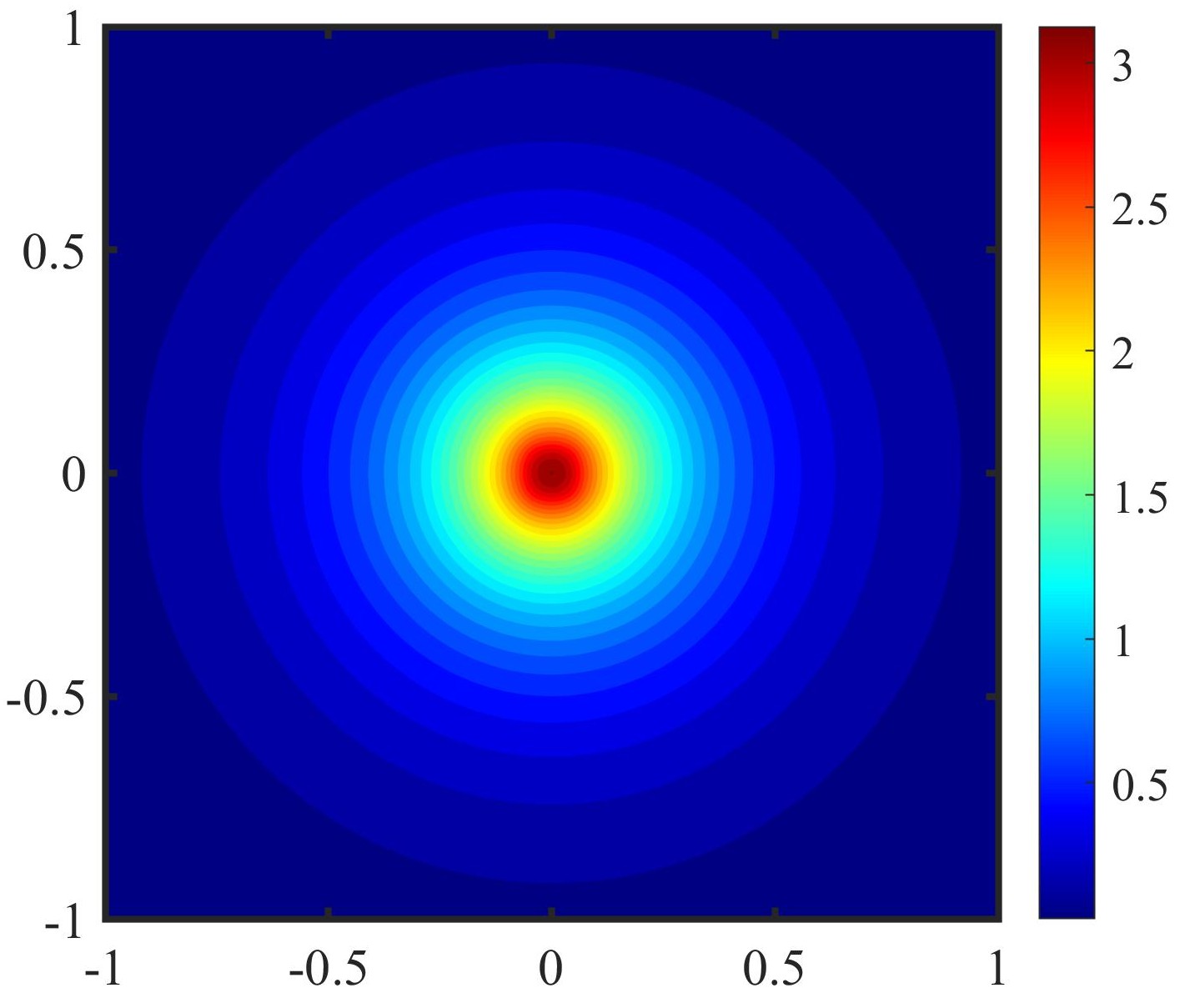}}
\caption{(Example 5) (a)  The electron density $\rho$ along $x$-axis, and (b) the contour plot of $\rho $ on the plane $z=0$ with $(x,y)\in[-1,1]^2$.} \label{fig:example_5_uh_1_p_2}
\end{figure}

We show the ground state energy obtained by the DG-IGA method
in Figure \ref{fig:example_5_Eigergy_Error}(a), from which we can clearly see that the energy is convergent to the reference energy as the initial mesh is successively refined. 
Figure \ref{fig:example_5_Eigergy_Error}(b) presents the convergence of the ground-state energy, showing that the convergence rate of the energy error is comparable to the square of the eigenfunction error in the DG-norm for the 3D linear eigenvalue problem, 
consistent with the energy error estimate in \cite{chen13}.

\begin{figure}[!ht]
  \centering 
  \subfigure[Computed ground state energy]{
 \includegraphics[width=0.47\textwidth]{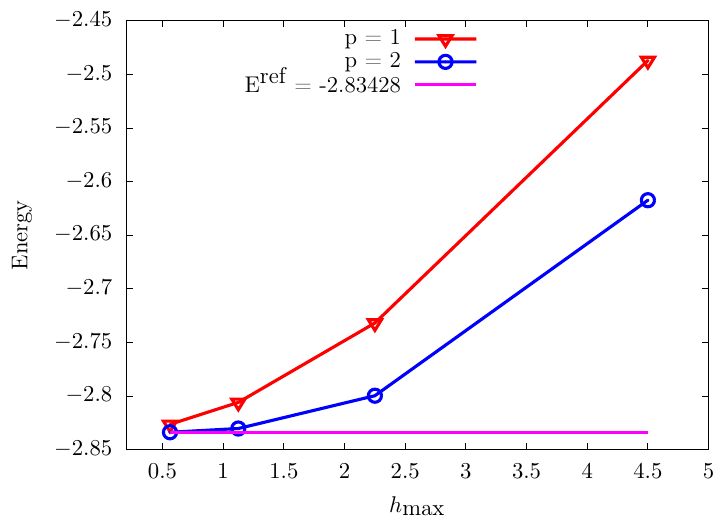} } 
 \subfigure[Error in energy]{
 \includegraphics[width=0.47\textwidth]{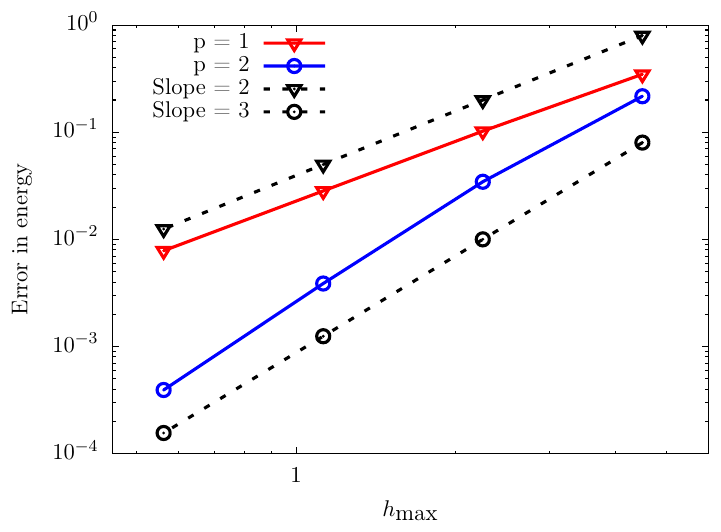} } 
  \caption{(Example 5) (a) The ground state energy computed by the DG-IGA method, and (b) the convergence of error in 
  the ground state energy.} \label{fig:example_5_Eigergy_Error}
\end{figure}

\section{Concluding remarks}
\label{sec-conclusions}

In this paper, we investigate the multi-patch DG-IGA approximation for full-potential electronic structure calculations. 
The computational domain is divided into several subdomains, where each part of the wavefunction is individually represented using B-spline functions, possibly with different degrees, over varying mesh sizes. 
These patches are then combined by DG methods. 
We provide an {\em a priori} error estimate of this DG-IGA approximation for linear eigenvalue problems, and present numerical experiments to support our theoretical results.
In addition, our error analysis suggests that higher-order convergence can be achieved through a multiscale refinement strategy adapted to the local regularity of the eigenfunctions (see Remark \ref{remark:multiscale-refinement}).

This work can serve as a unified analysis framework for the DG-IGA method applied to a class of elliptic eigenvalue problems, including those with discontinuous coefficients.
Beyond the accuracy and efficiency demonstrated in this paper, the DG scheme is also flexible and economical for adaptive procedures, as the nonconformity ensures that refinement is limited only to the subdomains where it is needed. An {\em a posteriori} error analysis and the development of an adaptive algorithm will be addressed in future work.

\section*{Acknowledgment}

X. Li would like to thank the support from the National Natural Science Foundation of China (No. 12301548) and the Open Project Program of Key Laboratory of Mathematics and Complex System (Grant No. K202302), Beijing Normal University.
The research of X. Meng was partially supported by the National Natural Science Foundation of China (Grant No. 12101057), 
the Scientific Research Fund of Beijing Normal University
(Grant No. 28704-111032105), and the Guangdong Higher Education Upgrading Plan (UIC-R0400024-21).

\appendix
\renewcommand\thesection{\appendixname~\Alph{section}}
\section{Proofs}
\label{append-proofs}
\renewcommand{\theequation}{A.\arabic{equation}}
\setcounter{equation}{0}

\subsection{Proof of Lemma \ref{lem:boundness}}
\label{append-boundness}

\begin{proof}
We first divide the left-hand side into three parts
\begin{eqnarray}\label{proof-r0}
	a^{\rm DG}(u, v) = I_1 + I_2 + I_3,
\end{eqnarray}
where
\begin{align*}
	&I_1 = \sum_{i=1}^N \int_{\Omega_i} \left( \frac{1}{2}\nabla u \cdot \nabla v + V uv  \right) \dd \vr, \\[1ex]
	&I_2 =  \frac{1}{2} \sum_{F_{ij}\in\facet} \int_{F_{ij}} \{\nabla u\} \cdot [v] \dd s
	+ \int_{F_{ij}} \{\nabla v\} \cdot [u] \dd s, \\[1ex]
	&I_3 = C_\sigma \sum_{F_{ij}\in\facet} \left(\frac{1}{h_i}+\frac{1}{h_j}\right) \int_{F_{ij}} [u] \cdot [v] \dd s.
\end{align*}
Since $V\in L^2(\Omega)$, we have
\begin{eqnarray}\label{proof-r1}
	I_1 \leq C \|u\|_{\rm DG} \sum_{i=1}^N \|v\|_{H^1(\Omega_i)}.
\end{eqnarray}
Using the trace inequality, $I_2$ can be estimated by
\begin{align}\label{proof-r2} 
	\nonumber
	& \hskip -0.2cm
	I_2 \leq C \sum_{F_{ij}\in\facet} \Big( \|[u] \|_{H^{\frac{1}{2}}(F_{ij})}\|\{\nabla v\}\|_{H^{-\frac{1}{2}}(F_{ij})}
	+ \|\{\nabla u\}\|_{H^{-\frac{1}{2}}(F_{ij})}\|[v]\|_{H^{\frac{1}{2}}(F_{ij})} \Big)
	\\[1ex] \nonumber
	&\leq C \sum_{F_{ij}\in\facet} \Big( \|[u]\|_{H^{\frac{1}{2}}(F_{ij})}(\|v\|_{H^{1}(\Omega_i)}+\|v\|_{H^{1}(\Omega_j)})
	+ (\|u\|_{H^{1}(\Omega_i)}+\|u\|_{H^{1}(\Omega_j)}) \|[v]\|_{H^{\frac{1}{2}}(F_{ij})} \Big)
	\\[1ex] 
	&\leq C \|u\|_{\rm DG} \|v\|_{\rm DG}.
\end{align}
Similarly, $I_3$ can be bounded by
\begin{eqnarray}\label{proof-r3}
	I_3 \leq C_\sigma \sum_{F_{ij}\in\facet} \left(\frac{1}{h_i}+\frac{1}{h_j}\right) \left\|[u] \right\|_{L^2(F_{ij})} \left\|[v] \right\|_{L^2(F_{ij})} \leq C \|u\|_{\dg} \|v\|_{\dg}.
\end{eqnarray}
Collecting \eqref{proof-r0} and the error bounds \eqref{proof-r1} to \eqref{proof-r3}, we prove the boundness \eqref{boundness}.
\end{proof}

\subsection{Proof of Lemma \ref{lem:coercive}}
\label{append-coercive}
\begin{proof}
Using the H\"{o}lder inequality, Sobolev's embedding theorem and Young's inequality, for any patch $\Omega_i$, we obtain that
\begin{multline*} 
	\qquad
	\left|\int_{\Omega_i} Vu^2 \dd\vr \right| 
	\leq \|V\|_{L^2({\Omega_i})}\cdot\|u^{\frac{1}{2}}\|_{L^4({\Omega_i})}\cdot \|u^{\frac{3}{2}}\|_{L^4({\Omega_i})}
	= C\|u\|^{\frac{1}{2}}_{L^2({\Omega_i})}\cdot\|u\|^{\frac{3}{2}}_{L^6({\Omega_i})}
	\\
	\leq C\|u\|^{\frac{1}{2}}_{L^2({\Omega_i})}\cdot\|u\|^{\frac{3}{2}}_{H^1({\Omega_i})}
	\leq C \left( \frac{\delta^{-4}}{4}\|u\|^2_{L^2({\Omega_i})}+\frac{3\delta^{\frac{3}{4}}}{4}\|u\|^{2}_{H^1({\Omega_i})} \right),
	\qquad
\end{multline*}
where $\delta>0$ is arbitrarily small.
Hence we have 
\begin{eqnarray}\label{coercive-veff}
\int_{\Omega_i} Vu^2 \dd\vr \geq -C\delta^{\frac{3}{4}}\|u\|^{2}_{H^1({\Omega_i})} -b\delta^{-4}\|u\|^2_{L^2({\Omega_i})}
\end{eqnarray}
with some positive constants $C$ and $b$.
Moreover, for any $F_{ij}\in\facet$, we have
\begin{align}
\label{coercive-jump}\nonumber
\left|\int_{F_{ij}}\{\nabla u\}\cdot[u]\dd s\right|
& \leq\|\{\nabla u\} \cdot {\vn}^{+}\|_{L^2(F_{ij})} \cdot
\|[u]\|_{L^2(F_{ij})}\\[1ex]\nonumber
& \leq C \Big(h_i^{-1} \|\nabla u\|^2_{L^2(\Omega_i)} + h_j^{-1} \|\nabla u\|^2_{L^2(\Omega_j)} \Big)^{\frac{1}{2}} \cdot
\|[u]\|_{L^2(F_{ij})}\\[1ex]\nonumber
&\leq \delta^2 \Big(\|\nabla u\|^2_{L^2(\Omega_i)} + \|\nabla u\|^2_{L^2(\Omega_j)} \Big)
+ \delta^{-2} (h_i^{-1}+h_j^{-1}) \|[u]\|^2_{L^2(F_{ij})}\\[1ex]
&\leq C\delta^2 (\|u\|^2_{H^1(\Omega_i)} + \|u\|^2_{H^1(\Omega_j)}) + \delta^{-2} (h_i^{-1}+h_j^{-1}) \|[u]\|^2_{L^2(F_{ij})},
\end{align}
where the trace inequality \eqref{trace-inequality-eq} is used in the second inequality.
Using \eqref{coercive-veff} and \eqref{coercive-jump}, we can derive \eqref{coercive} and complete the proof. 
\end{proof}

\subsection{Proof of Theorem \ref{thm:T-approximate}}
\label{append-T-approximation}
\begin{proof}
Denote $w=Tf$ and $w^{\rm DG}=T^{\rm DG}f$.
We define the projection $\hat{\Pi} w=\Pi_i w|_{\Omega_i}$ for $1\leq i\leq N$ and decompose the error $e=w-w^{\rm DG}$ as $e=\eta+\xi$, where
$\eta=w-\hat{\Pi} w$ and $\xi=\hat{\Pi} w-w^{\rm DG}$.
With simple calculations, we can easily obtain that $a^{\rm DG}(w,\xi)=(f,\xi)$, which leads to the property that $a^{\rm DG}(w-w^{\rm DG},\xi)=0$.
Using \eqref{coercive}, we have
\begin{eqnarray*}
	\|\xi\|^2_{\rm DG}\leq C a^{\rm DG}(\xi,\xi) = Ca^{\rm DG}(e-\eta,\xi) = -Ca^{\rm DG}(\eta,\xi).
\end{eqnarray*}
It follows from \eqref{boundness} that
\begin{eqnarray*}
\|\xi\|_{\rm DG} \leq C \|\eta\|_{\rm DG}.
\end{eqnarray*}
We obtain from \eqref{multi-patch-projection-err} that
if $u\in \widetilde{H}^{\vs}(\Omega)$, then we have for any $1\leq i\leq N$
\begin{eqnarray*}
\|\eta\|_{H^{1}(\Omega_i)}
\leq C h_i^{t_i-1} \|w\|_{H^{t_i}(\Omega_i)}
\end{eqnarray*}
with $t_i=\min\{p_i+1,s_i\}$.
Together with
\begin{eqnarray*}
\|[\eta]\|_{L^2(F_{ij})} \leq C \left( \|\eta\|_{H^{\frac{1}{2}}(\Omega_i)}+\|\eta\|_{H^{\frac{1}{2}}(\Omega_j)} \right)
\leq C \left( h_i^{t_i-\frac{1}{2}} \|w\|_{H^{t_i}(\Omega_i)} + h_j^{t_j-\frac{1}{2}} \|w\|_{H^{t_i}(\Omega_i)} \right),
\end{eqnarray*}
it leads to
\begin{eqnarray*}
\|w-w^{\rm DG}\|_{\rm DG} \leq \|\eta\|_{\rm DG}+\|\xi\|_{\rm DG} \leq C 
\sum_{i=1}^N  h_i^{t_i-1} \|Tf\|_{H^{t_i}(\Omega_i)}.
\end{eqnarray*}
The proof is complete.
\end{proof}

\subsection{Proof of Theorem \ref{thm:error_L2}}
\label{append-error-L2}
\begin{proof}
Denote $w=Tf$ and $w^{\rm DG}=T^{\rm DG}f$. 
Let $\varphi\in H^1_0(\Omega)$ be the solution to the dual source problem for $g\in L^2(\Omega)$ such that
\begin{eqnarray}
\label{dual-source}
a(v, \varphi) = (g, v)  \qquad \forall~v\in H^1_0(\Omega),
\end{eqnarray}
and $\varphi^{\rm DG}$ be the approximation of $\varphi$ in $V_{\vh}$ such that
\begin{eqnarray}
\label{dual-approximation}
a^{\rm DG}(v, \varphi^{\rm DG}) = (g, v)  \qquad \forall~v\in V_{\vh}.
\end{eqnarray}

According to Riesz representation theorem, we have
\begin{eqnarray}\label{riesz_L2_eq}
	\| w-w^{\rm DG}\|_{L^2(\Omega)} = \sup_{g\in L^2(\Omega), g\ne 0} \frac{(g, w-w^{\rm DG})}{\|g\|_{L^2(\Omega)}}.
\end{eqnarray}
We recall that $a^{\rm DG}(w-w^{\rm DG}, v)=0$ for any $v\in V_{\vh}$ in the proof of Theorem \ref{thm:T-approximate}.
It also suggests that $a^{\rm DG}(v, \varphi-\varphi^{\rm DG})=0$ for any $v\in V_{\vh}$ hold for the dual problem due to the symmetry of the bilinear form.
We then derive that
\begin{align}\nonumber
	(g, w-w^{\rm DG}) & = a^{\rm DG}(w, \varphi) - a^{\rm DG}(w^{\rm DG}, \varphi^{\rm DG})   \\[1ex]\nonumber
	& = a^{\rm DG}(w-w^{\rm DG}, \varphi) + a^{\rm DG}(w^{\rm DG}, \varphi - \varphi^{\rm DG}) \\[1ex]\nonumber
	& = a^{\rm DG}(w-w^{\rm DG}, \varphi - \varphi^{\rm DG}) + a^{\rm DG}(w-w^{\rm DG},  \varphi^{\rm DG})  +  a^{\rm DG}(w^{\rm DG}, \varphi - \varphi^{\rm DG}) \\[1ex]\label{dual-estimate-eq}
	& = a^{\rm DG}(w-w^{\rm DG}, \varphi - \varphi^{\rm DG}).
\end{align} 
Similar to \eqref{regularity-estimate-eq}, the regularity estimate $\displaystyle	\|\varphi\|_{H^2(\Omega)} \leq C \|g\|_{L^2(\Omega)}$ holds for the dual problem.
Using Theorem \ref{thm:T-approximate}, the regularity estimate, the boundness \eqref{boundness}, and \eqref{riesz_L2_eq}, we have
\begin{align}\nonumber
\| w-w^{\rm DG}\|_{L^2(\Omega)} & \leq C \sup_{g\in L^2(\Omega), g\ne 0} \frac{\|w-w^{\rm DG}\|_{\rm DG} \| \varphi - \varphi^{\rm DG}\|_{\rm DG}}{\|g\|_{L^2(\Omega)}} \\[1ex]\nonumber
& \leq C h_{\rm max} \sup_{g\in L^2(\Omega), g\ne 0} \frac{\|w-w^{\rm DG}\|_{\rm DG} 
	\| \varphi \|_{H^2(\Omega)}}{\|g\|_{L^2(\Omega)}} \\[1ex]\nonumber
& \leq C h_{\rm max}  \|w-w^{\rm DG}\|_{\rm DG},
\end{align}
which completes the proof.
\end{proof}

\subsection{Proof of Theorem \ref{thm:eigen-source-bound}}
\label{append-eigen-source-bound}
\begin{proof}
From the convergence of solution operator $\|T-T^{\dg}\|_{\mathscr{L}(L^2(\Omega))} \rightarrow 0$, we can obtain \eqref{source-eigenvalue-a} (see \cite{yang2010eigenvalue}).
Using the identity
\begin{eqnarray}\nonumber
(Tu_k^{\dg} - T^{\dg} u_k^{\dg}, u_k) = (u_k^{\dg} ,Tu_k) - (\lambda_k^{\dg})^{-1} (u_k^{\dg} , u_k) = \left( \lambda_k^{-1} - (\lambda_k^{\dg})^{-1} \right) (u_k^{\dg}, u_k),
\end{eqnarray}
we can derive
\begin{align}\nonumber
	\lambda^{\dg}_k - \lambda_k & = \frac{\lambda_k \lambda_k^{\dg}}{(u_k, u^{\dg}_k)} (Tu_k^{\dg} - T^{\dg} u_k^{\dg}, u_k) \\[1ex]\nonumber
	& = \frac{\lambda_k \lambda_k^{\dg}}{(u_k, u_k^{\dg})} \Big((T-T^{\dg})u_k, u_k) + ((T-T^{\dg})(u_k^{\dg}-u_k), u_k)) \Big)\\[1ex]\nonumber
	& = \frac{\lambda_k \lambda_k^{\dg}}{(u_k, u_k^{\dg})} ((T-T^{\dg})u_k, u_k) + R_1,
\end{align}
which leads to the convergence of eigenvalue.
From the symmetry of $T$ and $T^{\dg}$, $\lambda_k\rightarrow \lambda_k^{\dg} $ and \eqref{source-eigenvalue-a}, it follows that
\begin{align}\nonumber
	|R_1| & = \bigg| \frac{\lambda_k \lambda_k^{\dg}}{(u_k, u_k^{\dg})}  \left( (T-T^{\dg})(u_k^{\dg}-u_k), u_k \right) \bigg| \\[1ex]\nonumber
	& \leq C \Big| \left( (u_k^{\dg}-u_k), (T-T^{\dg}) u_k \right) \Big| \\[1ex]\nonumber
	& \leq C \|u_k^{\dg}-u_k\|_{L^2(\Omega)} \|(T-T^{\dg}) u_k \|_{L^2(\Omega)}  \leq C \| (T-T^{\dg})u_k \|^2_{L^2(\Omega)}.
\end{align}
Furthermore, let us recall the dual problems \eqref{dual-source} and \eqref{dual-approximation} with the source term $g=u_k$, which gives rise to the solution $\varphi=u_k/\lambda_k$. 
Similar to the derivations in \eqref{dual-estimate-eq}, we can obtain
\begin{equation*}
(Tu_k-T^{\dg}u_k, u_k) = a^{\dg}(Tu_k-T^{\dg}u_k, \varphi-v) \qquad \forall~ v\in V_{\vh}.
\end{equation*}
This leads to the estimate
\begin{equation*}
\left| (Tu_k-T^{\dg}u_k, u_k) \right| \leq C \|(T-T^{\dg})u_k\|_{\dg} \inf_{v\in V_{\vh}} \|u_k-v\|_{\dg},
\end{equation*}
which yields \eqref{source-eigenvalue-b}.

From the coerciveness and definition of $T^{\dg}$, we have 
\begin{align}\nonumber
	\| u^{\dg}_k - \lambda_k T^{\dg} u_k \|^2_{\dg} & \leq C a^{\dg}(u^{\dg}_k - \lambda_k T^{\dg} u_k, u^{\dg}_k - \lambda_k T^{\dg} u_k) \\[1ex]\nonumber
	& = C (\lambda_k^{\dg} u^{\dg}_k - \lambda_k u_k, u^{\dg}_k - \lambda_k T^{\dg} u_k) \\[1ex]\nonumber
	& \leq C \| \lambda^{\dg}_k u^{\dg}_k - \lambda_k u_k \|_{L^2(\Omega)}  
	\| u^{\dg}_k - \lambda_k T^{\dg} u_k \|_{L^2(\Omega)} \\[1ex]\nonumber
	& \leq C \left( \| \lambda^{\dg}_k u^{\dg}_k - \lambda_k u_k \|_{L^2(\Omega)}   + 
	\| u^{\dg}_k - \lambda_k T^{\dg} u_k \|_{L^2(\Omega)} \right)^2.
\end{align}
Using \eqref{source-eigenvalue-a}, \eqref{source-eigenvalue-b} and the fact $u_k=\lambda_k Tu_k$, we derive that
\begin{align}\nonumber
& \| u^{\dg}_k - \lambda_k T^{\dg} u_k \|_{\dg} \\[1ex]\nonumber
\leq & ~C \left( \| \lambda^{\dg}_k u^{\dg}_k - \lambda_k u_k \|_{L^2(\Omega)} + \| u^{\dg}_k - \lambda_k T^{\dg} u_k \|_{L^2(\Omega)} \right) \\[1ex]\nonumber
\leq & ~C \left( |\lambda^{\dg}_k-\lambda_k| 
+ |\lambda_k| \cdot \|u^{\dg}_k-u_k\|_{L^2(\Omega)} + \| u^{\dg}_k - u_k \|_{L^2(\Omega)}
+ \| u_k - \lambda_k T^{\dg} u_k\|_{L^2(\Omega)} \right) \\[1ex]
\leq & ~C \| (T-T^{\dg})u_k \|_{L^2(\Omega)}.
\label{process-eq}
\end{align}
We then split the error into two terms
\begin{eqnarray}\nonumber
\| u^{\dg}_k - u_k \|_{\dg} = |\lambda_k| \cdot \| Tu_k-T^{\dg}u_k \|_{\dg} + R_2.
\end{eqnarray}
Using the triangle inequality and \eqref{process-eq}, we can bound the residual term 
\begin{align}\nonumber
|R_2| & =  \Big| \|u^{\dg}_k-u_k \|_{\dg} - |\lambda_k| \cdot \| Tu_k-T^{\dg}u_k \|_{\dg}  \Big| \\[1ex]\nonumber
& \leq \Big| \|u^{\dg}_k-u_k \|_{\dg} - \| u_k- \lambda_k T^{\dg}u_k \|_{\dg}  \Big| \\[1ex]\nonumber
& \leq \| u^{\dg}_k - \lambda_k T^{\dg}u_k \|_{\dg} \leq C \| (T-T^{\dg})u_k \|_{L^2(\Omega)}.
\end{align}
The proof is complete.
\end{proof}

\bibliographystyle{plain}
\bibliography{ref}

\end{document}